\documentclass[pdflatex,sn-mathphys-num]{sn-jnl}



\usepackage{amsmath,amssymb,amsfonts,mathtools,bm}
\usepackage{graphicx,subcaption,booktabs,multirow,array}

\usepackage{siunitx}
\AtBeginDocument{\RenewCommandCopy\qty\SI}
\ExplSyntaxOn
\msg_redirect_name:nnn { siunitx } { physics-pkg } { none }
\ExplSyntaxOff
\usepackage{tikz,pgf,pgfplots}
\pgfplotsset{compat=1.18}
\usepackage{xurl}
\usepackage{cleveref}

\usepackage{amsmath, amssymb, amsfonts}
\usepackage{mathtools}
\usepackage{physics}
\usepackage{bm}
\usepackage{optidef}
\usepackage{calc}
\usepackage{lscape}            

\usepackage{float}

\theoremstyle{thmstyleone}    
\newtheorem{theorem}{Theorem}[section]
\newtheorem{lemma}{Lemma}[section]

\theoremstyle{thmstyletwo}    

\newtheorem{remark}[theorem]{Remark}

\newtheorem{assumption}{Assumption}[section]

\crefname{algocf}{Algorithm}{Algorithms}
\crefname{equation}{}{}
\Crefname{equation}{}{}
\crefname{assumption}{Assumption}{Assumptions}
\crefname{theorem}{Theorem}{Theorems}
\crefname{lemma}{Lemma}{Lemmas}
\crefname{proposition}{Proposition}{Propositions}
\crefname{corollary}{Corollary}{Corollaries}
\crefname{definition}{Definition}{Definitions}
\crefname{remark}{Remark}{Remarks}
\crefname{figure}{Figure}{Figures}
\crefname{table}{Table}{Tables}
\crefname{section}{Section}{Sections}
\crefname{subsection}{Subsection}{Subsections}

\newcommand{\RR}{\mathbb{R}} \newcommand{\NN}{\mathbb{N}}

\usepackage[linesnumbered,ruled,vlined]{algorithm2e}
\SetKwInput{KwInput}{Input}
\SetKwInput{KwOutput}{Output}
\SetKwBlock{DoWhile}{do}{while}
\SetKwComment{Comment}{$\triangleright$ }{}
\SetAlgoNlRelativeSize{-1}
\SetAlCapSty{textbf}
\SetAlCapNameFnt{\small}
\SetAlCapFnt{\small}
\SetKw{Break}{break}

\usepackage{tikz}
\usepackage{pgf}
\usepackage{pgfplots}
\pgfplotsset{
    compat=1.18,
    every axis/.append style={
        width=10cm,
        height=6cm,
        grid=both,
        grid style={line width=.1pt, draw=gray!10},
        major grid style={line width=.2pt,draw=gray!50}
    }
}
\usepgfplotslibrary{groupplots}

\makeatletter
\renewcommand{\orcidlogo}{\includegraphics[height=1.6ex]{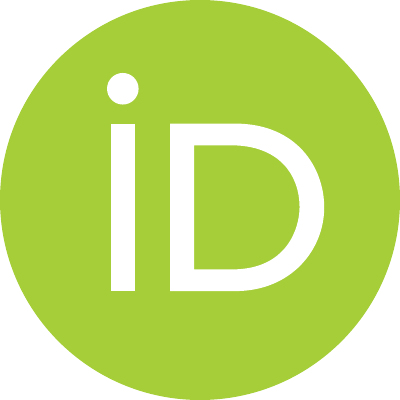}} 
\makeatother

\usepackage{enumitem}
\setlist[enumerate,1]{label=(\alph*)}

\begin{document}

\title[Inexact General Descent for DECO]{An Inexact General Descent Method with Applications in Differential Equation-Constrained Optimization}

\author*[1,2]{\fnm{Humberto Gimenes} \sur{Macedo}\orcid{https://orcid.org/0000-0002-0858-6283}}\email{gimenes.humberto@unifesp.br}
\author[1]{\fnm{Luís Felipe} \sur{Bueno}\orcid{https://orcid.org/0000-0002-8820-6606}}\email{lfelipebueno@gmail.com}

\affil[1]{\orgname{Universidade Federal de São Paulo (UNIFESP)},
\orgaddress{\street{Av. Cesare Mansueto Giulio Lattes, 1201 — Eugênio de Melo},
\city{São José dos Campos}, \state{SP}, \country{Brazil}}}

\affil[2]{\orgname{Instituto Tecnológico de Aeronáutica (ITA)},
\orgaddress{\street{Praça Marechal Eduardo Gomes, 50 — Vila das Acácias},
\city{São José dos Campos}, \state{SP}, \country{Brazil}}}

\abstract{
In many applications, gradient evaluations are inherently approximate, motivating the development of optimization methods that remain reliable under inexact first-order information. A common strategy in this context is adaptive evaluation, whereby coarse gradients are used in early iterations and refined near a minimizer. This is particularly relevant in differential equation–constrained optimization (DECO), where discrete adjoint gradients depend on iterative solvers. Motivated by DECO applications, we propose an inexact general descent framework and establish its global convergence theory under two step-size regimes. For bounded step sizes, the analysis assumes that the error tolerance in the computed gradient is proportional to its norm, whereas for diminishing step sizes, the tolerance sequence is required to be summable. The framework is implemented through inexact gradient descent and an inexact BFGS-like method, whose performance is demonstrated on a second-order ODE inverse problem and a two-dimensional Laplace inverse problem using discrete adjoint gradients with adaptive accuracy. Across these examples, adaptive inexact gradients consistently reduced optimization time relative to fixed tight tolerances, while incorporating curvature information further improved overall efficiency.
}


\maketitle

\section{Introduction}\label{sec:intro}

    Conventionally, convergence analyses of gradient-based algorithms assume the availability of exact information, encompassing both zeroth-order and first-order information. Nevertheless, in various settings, the acquisition of such information is subject to error. These errors can arise, for example, when the objective function lacks an analytic representation, exhibits stochasticity, is affected by finite-precision errors from floating-point arithmetic, or constitutes a smoothed approximation of an inherently non-smooth function. Moreover, errors can also stem from the tolerance-limited accuracy of iterative methods used to compute functional values, gradients, or broader search directions. This phenomenon is especially prevalent in differential equation-constrained optimization (DECO) problems. In this case, the objective function evaluation requires one or more numerical simulations, typically implemented through iterative procedures. Likewise, when the discrete adjoint method is used, computing the gradient normally involves the iterative solution of the adjoint equation. Consequently, in such contexts, both the objective function and its gradient are routinely subject to inexactness. Viewed more broadly, inexactness in numerical analysis arises along two complementary channels, namely (a) inexact solves, where subproblems or linear and nonlinear systems are solved only approximately, and (b) inexact information, where objective, constraint, and/or their gradients evaluations are themselves approximate.

    Research on nonlinear equations has predominantly focused on channel (a). Foundational work introduced and developed inexact Newton methods to reduce the computational cost associated with solving Newton equations, showing that exact solves are often unnecessary when the current iterate is far from the solution \cite{dembo_inexact_1982,eisenstat_globally_1994,eisenstat_choosing_1996}. This line of study has since been extended to quasi-Newton methods, where Hessian approximations replace the exact Jacobian and the resulting system is solved only approximately. Applications of such techniques to sparse nonlinear systems of equations have been explored \cite{bergamaschi_inexact_2001}, and subsequent studies have addressed issues of global convergence and constraints such as non-negativity \cite{birgin_globally_2003,arias_global_2020}. Recent investigations continue to advance the theory and practice of such inexact methods \cite{arias_inexact_2023,pietrus_inexact_2024}.

    From an optimization perspective, both channels~(a) and~(b) have been systematically investigated. Natural extensions of inexact Newton-type techniques to optimization arise from viewing first-order optimality systems as nonlinear equations \cite{martinez_1994}. In unconstrained problems, this leads to inexact Newton and quasi-Newton schemes, while for constrained problems, it gives rise to inexact Sequential Quadratic Programming (SQP) methods. Among these, inexact SQP methods have been extensively studied due to their broad applicability across diverse classes of problems. These methods interpret the resolution of first-order optimality conditions as a sequence of quadratic subproblems solved only approximately, thereby introducing inexactness primarily through channel~(a) \cite{leibfritz_inexact_1999,byrd_inexact_2008,curtis_inexact_2014}.    Inexactness of type~(a) also arises when approximately solving subproblems across several optimization frameworks, including higher-order methods \cite{birgin_2017}, interior-point strategies \cite{curtis_2010}, augmented Lagrangian formulations \cite{andreani_2008}, and in the broader class of inexact restoration techniques \cite{bueno_2020}.

    Regarding chanel (b),  trust-region and inexact restoration frameworks have been rcently developed to manage accuracy in optimization problems where inexactness is intrinsic, such as finite-sum minimization, in which the objective function is typically approximated via data subsampling to reduce computational cost \cite{curtis_2019,bellavia_2020,bellavia_2023,cao_2024}.
    Because such approaches enable the scalable training of large-scale AI models, this has become an active and rapidly growing area of research within the optimization community.
    
     On the other hand, the investigation of the effect of inexactness in traditional methods, which constitute the foundation of many modern optimization algorithms, dates back to a few decades earlier.
The analysis by \citet{polyak_introduction_1987} examined how gradient inaccuracies affect the steepest descent method in smooth, convex, unconstrained settings, and \citet{bertsekas_gradient_2000} later established convergence guarantees when the error is bounded by the step size times a gradient-dependent term.
    The accelerated framework of \citet{nesterov_method_1983} was extended by \citet{daspremont_smooth_2008} to inexact gradients without loss of optimal complexity.   
    
    \citet{devolder_first-order_2014} formalized the inexact first-order oracle and analyzed primal, dual, and accelerated methods with oracle errors in smooth and convex regimes. The analysis revealed that the accelerated gradient method accumulates errors linearly with the number of iterations, whereas the standard gradient method is not affected by this issue. Building on these findings, \citet{devolder_intermediate_2013} proposed the intermediate gradient method, which integrates the strengths of both prior approaches. It achieves a convergence rate comparable to the accelerated gradient method while maintaining the robustness of the standard gradient method in the presence of inexact oracles. The scope of the intermediate gradient method was subsequently broadened to address problems involving stochastic oracles in \cite{dvurechensky_stochastic_2016}and \citet{ghadimi_accelerated_2016}. This analysis was further expanded to include constrained optimization problems \cite{ghadimi_mini-batch_2016}.

     \citet{khanh_inexact_2023} developed an inexact reduced gradient method for nonconvex, continuously differentiable problems, both with and without Lipschitz continuous gradients, establishing convergence under various step-size strategies. This research direction was further explored in \citet{khanh_new_2024}, where a new inexact gradient method for smooth functions was introduced.

    Inexact trust-region and multilevel SQP algorithms have also been developed for  DECO applications, where both the state and adjoint equations, as well as the quadratic subproblems, are solved only approximately, with adaptive mechanisms controlling inexactness and discretization accuracy \cite{heinkenschloss_inexact_2008,ziems_adaptive_2011}. In the same context, \citet{brown_adaptive_2018,brown_inexactly_2018,brown_effect_2022} analyzed how inexactness in discrete adjoint computations impacts gradient methods. They argue that when the state and adjoint equations are solved with an adaptive tolerance proportional to the gradient norm, the inexact gradient method achieves the same convergence rate as its exact counterpart. In these works there is a connection between inaccuracies in solving systems of equations and in the gradient of the functions of the optimization problem.

    As we can notice, inexact gradient-based algorithms have become increasingly important, driven by both theoretical challenges and practical demands. In response, we introduce a novel inexact general descent method and demonstrate its applications to DECO problems, a domain where inexactness is inherent and remains a central focus of research. Our approach provides convergence proofs for an inexact version of classical general directions strategies under both bounded and diminishing step-size rules. The analysis focuses on these more restricted step-size strategies because, in DECO settings, line-search procedures are not readily applicable due to inexactness in the objective function. 
    
    Among the works discussed so far, the most closely related are \cite{bertsekas_gradient_2000}, \cite{devolder_first-order_2014}, \cite{brown_effect_2022}, and \cite{khanh_inexact_2023}, yet our contributions differ in several key aspects. \citet{bertsekas_gradient_2000} formulate the inexact conditions for the general direction using the unknown exact gradients, whereas our method applies these conditions directly with the available inexact gradients. The focus of \citet{devolder_first-order_2014} is convex, first-order (gradient) settings with fixed search directions, rather than flexible direction schemes for general smooth problems, and it does not address DECO applications. An unconventional strategy is adopted in \citet{brown_effect_2022}, where the convergence is done assuming a strong connection between the iterates generated by an exact algorithm and its inexact counterpart. In contrast, we follow the traditional approach by directly proving the convergence of  the  inexact methods. The analysis in \citet{brown_effect_2022} also leaves the step-size rules unspecified and restricts some results to the gradient method.  Finally, the search direction in \citet{khanh_inexact_2023} is also limited to being proportional to the inexact gradient, whereas our approach admits fully general directions. Algorithmically, our method is also simpler, arising as a straightforward extension of classical general descent methods.

    The remainder of the paper is organized as follows. \cref{sec:deco_fram} reviews the DECO setting, its applications, standard solution approaches, and alternative formulations. \cref{sec:std_cas} introduces two inverse DECO case studies, specifying the objectives and the discretizations of the governing equations. \cref{sec:conv_the} states preliminary mathematical facts and then develops the proposed inexact general descent framework with a convergence analysis. \cref{subsec:adj_error} details the adaptive test-and-tighten scheme used to compute discrete adjoint gradients with known and controlled accuracy. \cref{sec:num_res} reports numerical experiments on the two case studies and evaluates the effectiveness of the methods. Finally, \cref{sec:fin_rem} summarizes the main findings and outlines directions for future research.

\section{DECO Framework and Formulations}\label{sec:deco_fram}
    DECO combines high-fidelity numerical simulations with optimization algorithms to solve problems governed by differential equations \cite{hicken_arnoldi-based_2015,antil_frontiers_2018}. Applications include inverse problems, design optimization, and optimal control. Inverse problems seek to identify unknown parameters in the governing equations from limited measurements of the system state, such as estimating thermal conductivity from sparse temperature data. Design optimization aims to enhance system performance through systematic modifications of geometry, configuration, or operating conditions. A classical example is aerodynamic shape optimization, where the goal may be to reduce drag or maximize the lift-to-drag ratio of an aircraft wing. Optimal control focuses on determining control inputs that steer a system toward a desired state while satisfying physical and operational constraints and optimizing a performance criterion.  
    
    To address such applications, two principal approaches to DECO have been developed: optimize--then--discretize (OD) and discretize--then--optimize (DO) \cite{liu_non-commutative_2019}. In the OD approach, one derives continuous optimality conditions, such as the Karush--Kuhn--Tucker (KKT) conditions, and subsequently discretizes them for numerical implementation. By contrast, the DO approach first discretizes the governing equations and then derives the optimality conditions directly from the discrete system. These two approaches are not equivalent in general and can lead to distinct solutions \cite{collis_analysis_2024}.  
    
    Within the DO approach, the control variables $\vb{z}\in\RR^m$ and state variables $\vb{y}\in\RR^n$ are coupled through the discrete governing equations
    \begin{equation}
        \label{eq:R=0}
        \vb{R}(\vb{z}, \vb{y}) = \vb{0}_{\RR^n},
    \end{equation}
    where $\vb{R} : \RR^m \times \RR^n \to \RR^n$ denotes the residual vector of the discretized system. The equations in \cref{eq:R=0} are also known as state equations. Depending on how \cref{eq:R=0} is enforced, two distinct formulations arise \cite{haftka_simultaneous_1985,cramer_problem_1994}.  
    
    In the simultaneous analysis and design (SAND) formulation, both the state and control variables are treated as unknowns, and the state equation is imposed as a constraint. This yields the following full-space optimization problem:
    \begin{mini*}
        {}{F(\mathbf{z}, \mathbf{y}), \qquad \mathbf{z} \in \mathbb{R}^m, \ \mathbf{y} \in \mathbb{R}^n,}
        {}{} 
        \addConstraint{\mathbf{R}(\mathbf{z}, \mathbf{y})}{=\mathbf{0}_{\mathbb{R}^n},}
    \end{mini*}
    where $F:\mathbb{R}^{m}\times\mathbb{R}^n \to \mathbb{R}$ is a continuously differentiable objective function.  
    Although SAND provides a unified framework, its application to high-fidelity simulations is often hindered by the computational cost and memory demands of treating both the state and control variables simultaneously \cite{hicken_comparison_2013}.

    Despite the state equation may not admit a unique solution in general \cite{jameson_airfoils_1991}, under suitable regularity conditions the implicit function theorem guarantees local uniqueness of the state for a given control \cite{borggaard_computational_1998}. In this context, the state variables can be eliminated, which leads to the nested analysis and design (NAND) formulation. In this work, we focus on residual functions that are affine in the state variable and associated with well-conditioned, nonsingular matrices, so that a global mapping can be defined. Thus, we can find a mapping
    \begin{equation*}
        \vb{u} : \mathbb{R}^m \to \mathbb{R}^n
        \quad \text{ such that } \quad
        \vb{R}\big(\vb{z}, \vb{u}(\vb{z})\big) = \vb{0}_{\mathbb{R}^n},
    \end{equation*}
    and, by substituting $\vb{y} = \vb{u}(\vb{z})$ into the objective function, we obtain a reduced-space optimization problem
    \begin{mini}
      {\vb{z} \in \mathbb{R}^m}
      {\widetilde{F}(\vb{z})
        = F\left(\vb{z},\vb{u}(\vb{z})\right)}
      {\label{eq:opt_p_2}}{},
    \end{mini}
    where $\widetilde{F}: \mathbb{R}^{m} \to \mathbb{R}$ denotes the reduced objective function, assumed to be $L$-smooth and bounded from below in this work.  This formulation closely resembles the classical Generalized Reduced Gradient (GRG) methods \cite{abadie_GRG_1969,lasdon_1978} when the state variables are computed exactly. In practice, however, this is rarely the case, since the state dependence on the controls is obtained through numerical solvers, introducing inexactness into both the objective and its gradient evaluations.
    SAND and NAND are therefore fundamentally distinct and need not produce equivalent solutions \cite{heinkenschloss_interface_1999}.  

    With the formulation established, attention turns to solution strategies. Approaches fall broadly into two classes: gradient-free and gradient-based methods \cite{zingg_comparative_2008}. Gradient-free approaches rely solely on function evaluations and are effective for nonsmooth or noisy objectives, but their high computational cost typically restricts them to small- or medium-scale problems \cite{holst_aerodynamic_2001,chernukhin_multimodality_2013,he_shape_2014}. Gradient-based methods, by contrast, exploit derivative information, scale efficiently to high-dimensional control spaces, and are thus the preferred choice for large-scale DECO \cite{secco_component-based_2018,secco_rans-based_2019}. Their performance, however, depends critically on the accurate and efficient computation of sensitivities. Several techniques are available for this purpose, including finite differences, complex-step differentiation \cite{squire_using_1998}, adjoint methods \cite{bryson_applied_1975}, and algorithmic differentiation (AD) \cite{wengert_simple_1964}. Among these, adjoint methods are particularly advantageous in large-scale settings because their cost is essentially independent of the number of control variables \cite{kenway_effective_2019}. In practice, hybrid adjoint--AD strategies are increasingly adopted to automate derivative computation and further enhance efficiency \cite{mader_adjoint_2008}.

    As noted above, adjoint methods enable scalable gradient-based optimization. In fact, the discrete adjoint method arises naturally when deriving the first-order necessary optimality conditions for \cref{eq:opt_p_2}. Specifically, applying the chain rule to $\widetilde{F}$ yields
    \begin{equation}
        \label{eq:opt_nand}
        \grad \widetilde{F}(\vb{z}) \equiv 
        \grad_{\vb{z}} F \!\left(\vb{z},\vb{u}(\vb{z})\right) -  
        \left[\pdv{\vb{R}}{\vb{z}}\bigg\rvert_{(\vb{z}, \vb{u}(\vb{z}))}\right]^\top \bm{\psi}
        = \vb{0}_{\mathbb{R}^m},
    \end{equation}
    where $\pdv{\vb{R}}{\vb{z}} : \RR^{m}\times\RR^{n} \to \RR^{\,n \times m}$ and  $\grad_{\vb{z}} F : \RR^{m}\times\RR^{n} \to \RR^{m}$ represent the partial Jacobian and gradient of $\vb{R}$ and $F$ with respect to $\vb{z}$, respectively, and the adjoint variable $\bm{\psi} \in \RR^n$ is obtained by solving the so called adjoint equation
    \begin{equation}
        \label{eq:adj_eq}
        \vb{R}_{\psi}\!\left(\boldsymbol{\psi};\vb{z}, \vb{u}(\vb{z})\right) \equiv 
        \left[\pdv{\vb{R}}{\vb{y}}\bigg\rvert_{(\vb{z}, \vb{u}(\vb{z}))}\right]^\top \bm{\psi} 
        - \grad_{\vb{y}} F \!\left(\vb{z},\vb{u}(\vb{z})\right)
        = \vb{0}_{\mathbb{R}^n}.
    \end{equation}
    In this context, $\pdv{\vb{R}}{\vb{y}}:\RR^{m}\times\RR^{n}\to\RR^{n\times n}$ and $\grad_{\vb{y}} F:\RR^{m}\times\RR^{n}\to\RR^n$ denote the partial Jacobian and gradient with respect to $\vb{y}$. Since the state and adjoint equations are enforced in the exact NAND formulation, the optimizer needs only to drive \eqref{eq:opt_nand} to zero.

    In summary, the discrete adjoint method provides an efficient and scalable way to compute gradients in DECO problems. However, in practical implementations, the state equation is often not solved exactly, especially during the early iterations when high accuracy is unnecessary. Consequently, the state $\vb{u}(\vb{z})$ is not computed exactly, and the quantities involved in the gradient and adjoint equations \eqref{eq:opt_nand}-\eqref{eq:adj_eq} become inexact. In addition, the adjoint equation may itself be solved only approximately, introducing another source of error. As a result, the gradient $\grad \widetilde{F}(\vb{z})$ is evaluated at inexact state and adjoint variables, leading to an inexact gradient. These observations indicate that inexact gradient-based methods naturally arise as an appropriate and effective framework for differential-equation-constrained optimization (DECO) problems.

\section{Case Studies}\label{sec:std_cas}

    This section presents two study cases, each formulated as an inverse problem consistent with the NAND formulation \cref{eq:opt_p_2}. In the first case, which is detailed in \cref{subsec:cs_1}, the governing equation is a second-order ODE. The second case, described in \cref{subsec:cs_2}, uses the Laplace equation as its governing equation. These particular governing equations were chosen because they lead to state and adjoint equations with a sparse tridiagonal structure. Such systems can be efficiently solved by iterative methods terminated when the residuals meet a specified tolerance, giving explicit control of the state and adjoint residuals and, consequently, adjoint-based gradients of known accuracy, as discussed in \cref{subsec:adj_error}. The controls, which are either boundary data or coefficients, are adjusted to fit measurements at selected locations. In both study cases, the associated state equation is linear in the state variable and can be expressed in the following affine form:
    \begin{equation}
        \label{eq:ge_sys}
        \vb{R}(\vb{z}, \vb{y}) = \vb{A}(\vb{z})\,\vb{y} - \vb{b}(\vb{z}),
    \end{equation}
    where $\vb{A}:\RR^{m}\to\RR^{n\times n}$ is a matrix-valued function determined by the governing equation and the finite-difference stencil used for discretization, and $\vb{b}:\RR^{m}\to\RR^{n}$ is a vector-valued function that incorporates boundary conditions and may also include source terms.
        
    \subsection{Case Study 1: Inverse Problem with Parametrized Coefficients and Dirichlet Boundary Conditions}
        \label{subsec:cs_1}

        In the first case study, the governing equation consists of a boundary value problem (BVP) defined by a second-order ODE with Dirichlet boundary conditions and a polynomial source term, as follows:
        \begin{equation}
            \label{eq:bvp_edo}
            \begin{split}
                &z_0\dv[2]{u}{x} + z_1\dv{u}{x} + z_2u = \underbrace{z_3 + z_4x + z_5x^2}_{p(x)}, \qquad x \in (0,1), \\
                &u(0) = z_6, \qquad u(1) = z_7, 
            \end{split}
        \end{equation}
        where $\vb{z} = \begin{bmatrix} z_0 & z_1 & \cdots & z_7 \end{bmatrix} \in \mathbb{R}^8$ denotes the vector of control variables, with $z_0 \neq 0$. This ODE commonly appears in the analysis of second-order oscillatory systems, including mechanical (mass-spring) and electrical (RLC) systems.
        
        Discretizing \cref{eq:bvp_edo} with a Cartesian mesh of $M > 0$ interior points, mesh spacing $\Delta x = 1 / (M+1)$, and second-order central difference operators, leads to a state equation in which $\vb{A}: \RR^8\to\RR^{(M+2) \times (M+2)}$ and $\vb{b}: \RR^8\to\RR^{(M + 2)}$ are defined as follows:
         \[
            \vb A(\vb z)=
                \begin{bmatrix}
                    1 & \vb{0}_{\RR^{M}}^{\!\top} & 0\\[4pt]
                    \overline{a}\,\vb e_{1} &
                    \operatorname{tridiag}\!\big(\overline{a},\,\overline{b},\,\overline{c}\big) &
                    \overline{c}\,\vb e_{M}\\[4pt]
                    0 & \vb{0}_{\RR^{M}}^{\!\top} & 1
                \end{bmatrix},
                \quad \text{and} \quad
            \vb{b}(\vb{z}) =
                \begin{bmatrix}
                    z_6 \\
                    p(x_1)  \\
                    p(x_2) \\
                    \vdots \\
                    p(x_{M-1}) \\
                    p(x_M) \\
                    z_7
                \end{bmatrix},
        \]
        respectively, with 
        \begin{equation*}
            \overline{a} = \left(\dfrac{z_0}{\Delta x^2} - \dfrac{z_1}{2\Delta x}\right), \quad \overline{b} = \left(-\dfrac{2z_0}{\Delta x^2} + z_2\right), \quad \text{and} \quad \overline{c} = \left(\dfrac{z_0}{\Delta x^2} + \dfrac{z_1}{2\Delta x}\right).
        \end{equation*}
       Here, $\vb{e}_1$ and $\vb{e}_M$ denote the first and last canonical basis vectors of $\RR^M$. If $\vb{A}(\vb{z})$ is strictly diagonally dominant, the Levy–Desplanques theorem ensures the matrix is non-singular, thus guaranteeing a unique solution to the state equation. For this problem, given $z_1 \neq 0$ and $M > 0$, it is sufficient to choose $z_0 \geq \frac{|z_1|}{2(M + 1)}$ and any $z_2 < 0$ to obtain strict diagonal dominance. The other coefficients do not affect this property. Thus, $\vb{u}: \RR^8 \to \RR^{M+2}$ is defined as $\vb{u}(\vb{z}) = [\vb{A}(\vb{z})]^{-1}\vb{b}(\vb{z})$. 

        Regarding the objective function and the reference values, we proceed as follows. 
        First, a reference control vector $\vb{z}_{\text{ref}} \in \mathbb{R}^8$ is defined and used to generate the corresponding reference state values. 
        Specifically, the analytical solution of \cref{eq:bvp_edo}, denoted by $u(x;\vb{z})$, is evaluated at $\vb{z}_{\text{ref}}$, and the reference values are set as
        \[
            u_k^* = u(x_k;\vb{z}_{\text{ref}}), \quad \forall\, k \in S,
        \]
        where $S \subset \{0, 1, \dots, M+1\}$ is a subset of mesh indices obtained by uniformly sampling $n_P \in (0, M + 1)$ points. 
        The objective function $\widetilde{F}: \mathbb{R}^8 \to \mathbb{R}$ is then defined as the sum of squared residuals between the computed solution and the reference values, both evaluated at the indices in $S$:
        \begin{align}
            \widetilde{F}(\vb{z})
            &= \sum_{k \in S} \left[ u_k(\vb{z}) - u_k^* \right]^2 
            + \alpha \left[u_0(\vb{z}) - 1\right]^2, 
            \label{eq:ode_fo}
        \end{align}
        where $u_k:\RR^8\to \RR$ denotes the $k$-th component of $\vb{u}(\vb{z})$, and
        $F:\mathbb{R}^{M+10} \to \mathbb{R}$ is defined as 
        $F(\vb{z}, \vb{y}) = \sum_{k \in S} [y_k - u_k^*]^2$. 
        The last term in \cref{eq:ode_fo}, weighted by the parameter $\alpha > 0$, penalizes deviations of $u_0(\vb{z})$, the leading coefficient, from unity, ensuring consistency with the normalized form of the second-order ODE.

        For completeness, \cref{fig:cs1_recon} illustrates the solution of a specific instance of this inverse problem on a representative grid ($M=64$). The optimized curve aligns precisely with the analytical solution at the sampled reference points.
        \begin{figure}[H]
            \centering
            \includegraphics[scale=0.7]{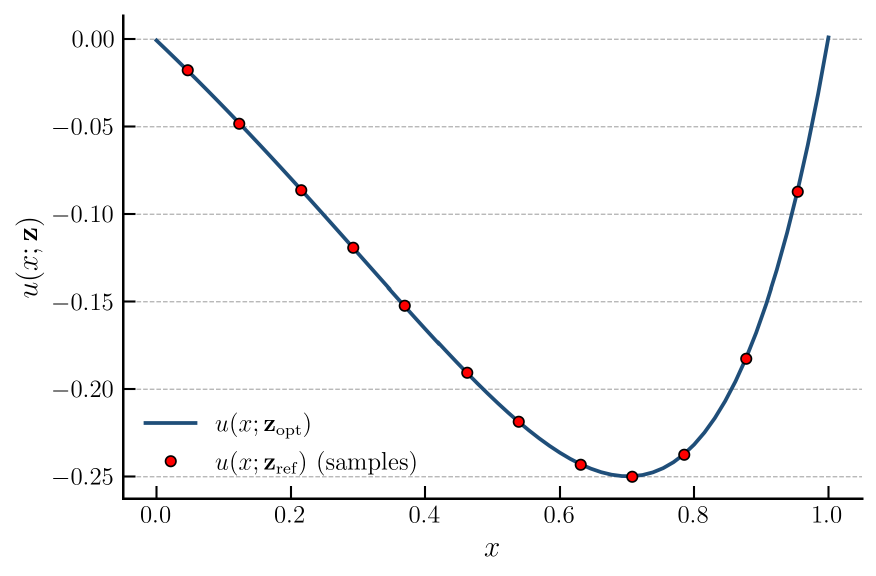}
            \caption{Optimized solution curve with reference points.}
            \label{fig:cs1_recon}
        \end{figure}
        
        \subsection{Case Study 2: Inverse Problem with Parametrized Boundary Conditions}
            \label{subsec:cs_2}
            
            In the second case study, the governing equation is a BVP defined by Laplace's equation on a unit square, subject to Dirichlet boundary conditions, as follows:
            \begin{equation}  
                \label{eq:bvp_lap}
                \begin{alignedat}{2}
                    \pdv[2]{u}{x} + \pdv[2]{u}{y} &= 0, \quad &&(x, y) \in \Omega, \\
                    u(x, y) - u_0(x,y; \vb{z}) &= 0, \quad &&(x, y) \in \partial \Omega,
                \end{alignedat}
            \end{equation}
            where $\vb{z} = \begin{bmatrix}z_0 & z_1& z_2& \cdots& z_\ell\end{bmatrix}^T \in \mathbb{R}^{\ell + 1}$, with $\ell >1$, and $\Omega = \{(x,y) \in \mathbb{R}^2 \mid x, y \in (0, 1)\}$ is the unit square domain with boundary $\partial\Omega = \partial\Omega_1 \cup \partial\Omega_2 \cup \partial\Omega_3 \cup \partial\Omega_4$. These boundary subsets are defined by
            \begin{alignat*}{3}
                \partial\Omega_1 &= \{(x, 0) \in \RR^2 \mid x \in (0, 1)\}, \quad 
                &&\partial\Omega_2 &&= \{(x, 1) \in \RR^2 \mid x \in (0, 1)\},\\
                \partial\Omega_3 &= \{(0, y) \in \RR^2 \mid y \in [0, 1]\},
                \quad 
                &&\partial\Omega_4 &&= \{(1, y) \in \RR^2 \mid y \in [0, 1]\}.
            \end{alignat*}
            The function $u_0: \partial \Omega \to \mathbb{R}$ defines the shape of the Dirichlet boundary conditions. It is expressed as
            \begin{equation*}
                u_0(x, y; \vb{z}) =
                \begin{cases}
                    0, &(x,y) \in \partial\Omega_1 \cup \partial\Omega_2 \cup \partial\Omega_3, \\
                    f(y;\vb{z}), & (x, y) \in \partial\Omega_4,
                \end{cases}
            \end{equation*}
            where $f:[0,1]\to \mathbb{R}$ is a polynomial function given by
            \begin{equation*}
            \label{po_eq:f}
            f(y;\vb{z}) =  z_0 + z_1y + z_2y^2 + \dots + z_\ell y^\ell.
            \end{equation*}
            It is important to note that \cref{eq:bvp_lap} arises when modeling steady-state temperature distributions in thermodynamic systems.
            
            The BVP described in \cref{eq:bvp_lap} can be discretized using a uniform two-dimensional Cartesian mesh with $M > 0$ interior points in each direction and mesh spacing $h = 1 / (M + 1)$. Since Laplace's equation is elliptic, central difference operators are suitable for approximating the derivatives. By organizing the algebraic equations resulting from the discretization in a row-wise fashion, we obtain a state equation for which $\vb{A} : \RR^{\ell + 1} \to \RR^{(M+2)^2 \times (M+2)^2}$ and $\vb{b} : \RR^{\ell + 1} \to \RR^{(M+2)^2}$ are defined as follows:
            \begin{equation*}
                \vb{A}(\vb{z}) =
                 \left[
                \begin{array}{c|cccccc|c}
                    \vb{I} &&&&&&& \\
                    \hline
                    \vb{B} & \vb{T} & \vb{B} &&&&&\\
                    & \vb{B} & \vb{T} & \vb{B}&&&&\\
                    & & \vb{B} & \vb{T} & \vb{B} & &&\\
                    &&&\ddots &\ddots& \ddots&&\\
                    &&&&\vb{B}&\vb{T}&\vb{B}&\\
                    &&&&&\vb{B}&\vb{T}&\vb{B}\\
                    \hline
                    &&&&&&&\vb{I}\\
                \end{array}
                \right], \quad \text{ and } \quad  
                \vb{b}(\vb{z}) = 
                \begin{bmatrix}
                    \vb{0}_{(M+1)(M+2)\times 1} \\
                    \cmidrule(lr){1-1}
                    f(0;\vb{z})\\
                    f(h;\vb{z}) \\
                    \vdots \\
                    f((M+1)h;\vb{z})
                \end{bmatrix},
            \end{equation*}
            respectively, where
            \[
                \vb T =
                \begin{bmatrix}
                    1 & \vb{0}_{\RR^{M}}^{\!\top} & 0\\[4pt]
                    \vb e_{1} & \operatorname{tridiag}(1,-4,1) & \vb e_{M}\\[4pt]
                    0 & \vb{0}_{\RR^{M}}^{\!\top} & 1
                \end{bmatrix},
                \quad \text{and} \quad
                \vb B =
                \begin{bmatrix}
                    0 & \vb{0}_{\RR^{M}}^{\!\top} & 0\\[4pt]
                    \vb{0}_{\RR^{M}} & \vb I_{M\times M} & \vb{0}_{\RR^{M}}\\[4pt]
                    0 & \vb{0}_{\RR^{M}}^{\!\top} & 0
                \end{bmatrix}
            \]
            are in $\RR^{(M+2)\times(M+2)}$. Here, $\vb{I}_{M\times M}$ is the identity matrix of order $M$. It should be noted that if the boundary conditions are addressed separately, $\vb{A}$ can be simplified to a symmetric, negative definite block tridiagonal matrix. This property guarantees that $\vb{A}$ is non-singular.  As a result, we can define $\vb{u}: \RR^{\ell +1} \to \RR^{(M+2)^2}$ by $\vb{u}(\vb{z}) = \left[\vb{A}(\vb{z})\right]^{-1}\vb{b}(\vb{z})$ as in the first case study.

            Turning to the construction of the objective function and the reference values for this problem, we proceed as follows. 
            Initially, we define $\Omega_d \subset \Omega$ as the set of desired point coordinates used to evaluate each reference value. 
            To construct this set, we employ the Latin hypercube sampling technique, which enables the generation of a uniformly distributed set of points across the domain \cite{mckay_comparison_1979}. 
            Next, these points are mapped to the corresponding mesh points to build $\Omega_d$, and we outline the set of its mesh indices as
            \[
                \mathcal{I}(\Omega_d) = \{(j, k) \in \mathcal{I}\times\mathcal{I} \mid (jh, kh) \in \Omega_d\},
                \quad \text{where} \quad \mathcal{I} = \{0,1,\ldots,M+1\}.
            \]
            For the evaluation of the reference values, we consider the analytical solution of \cref{eq:bvp_lap} at a reference vector of control variables $\vb{z}_{\text{ref}}\in\RR^{\ell + 1}$, denoted by $u(x,y;\vb{z}_{\text{ref}})$, obtained by solving the problem with $ f(y;\vb{z}_{\text{ref}})$.
            This analytical solution is then used to compute the reference state values $u_d^{j,k}$ at the points in $\Omega_d$, so that
            \[
                u_d^{j,k} = u(jh,kh;\vb{z}_{\text{ref}}), \quad \forall\,(j,k)\in \mathcal{I}(\Omega_d).
            \]
            Then, we finally define the objective function $\widetilde{F}:\RR^{\ell + 1}\to\RR$ as
            \begin{align*}
                \widetilde{F}(\vb{z})
                &= \sum\limits_{(j, k) \in \mathcal{I}(\Omega_d)} \left[u_{j, k}(\vb{z}) - u_{d}^{j, k}\right]^2,
            \end{align*}
            where $u_{j,k}: \RR^{\ell + 1} \to \RR$ denotes the $(j,k)$-th component of $\vb{u}(\vb{z})$, and $ F:\RR^{\ell + 1 + (M+2)^2}\to\RR$ is defined as $F(\vb{z}, \vb{y}) = \sum\limits_{(j, k) \in \mathcal{I}(\Omega_d)} \left[y_{j, k} - u_{d}^{j, k}\right]^2$.

            For completeness, \cref{fig:cs2_recon} illustrates the solution of a specific instance of this inverse problem in a representative grid ($M=64$). The optimized surface aligns precisely with the analytical solution at the sampled reference points.
            \begin{figure}[H]
                \centering
                \includegraphics[scale=0.16]{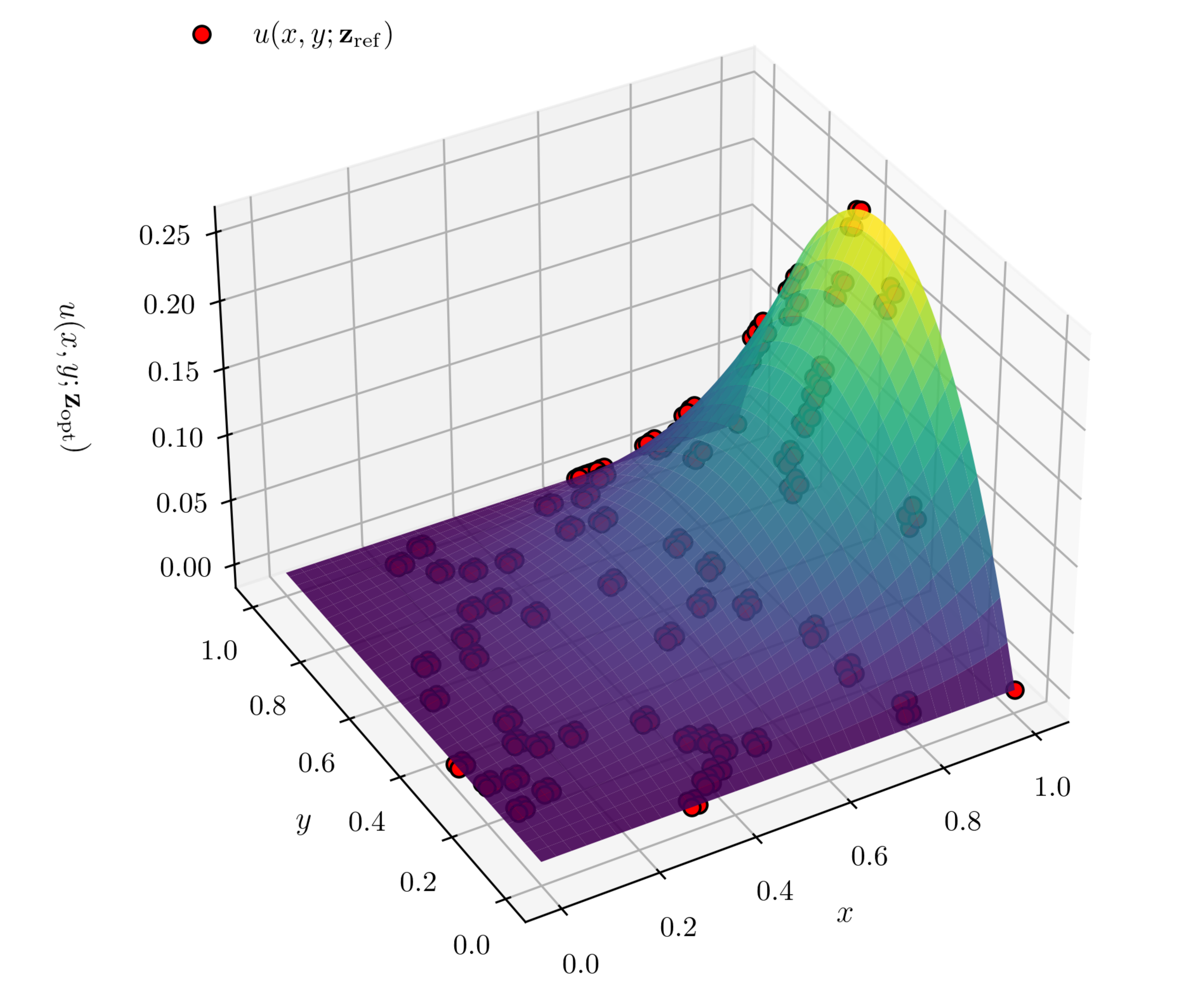}
                \caption{Optimized solution surface with reference points.}
                \label{fig:cs2_recon}
            \end{figure}

\section{Convergence Theory}\label{sec:conv_the}

        This section begins with preliminaries that establish the necessary background for the analysis, which will be restricted to the case of smooth functions. We then introduce the proposed algorithm and outline its key properties. Following this, we determine conditions under which we carry out a convergence analysis.

        \subsection{Mathematical Preliminaries}

            It is well known that smooth functions can be upper-bounded by a quadratic polynomial \cite{beck_first-order_2017}. In a series of works, \citet{devolder_stochastic_2011, devolder_exactness_2013, devolder_first-order_2014} demonstrated a similar result for inexact gradients, as shown in the following lemma. The proof provided here follows Claim 4.1 in \cite{vasin_accelerated_2023}.
            \begin{lemma}
                \label{lem:iqub}
                 Let $\widetilde{F}: \mathbb{R}^m \to \mathbb{R}$ be an $L$-smooth function, and $\vb{z} \in \mathbb{R}^m$ a specified vector of control variables. Suppose that the inexact gradient of $\widetilde{F}$, denoted by $\overline{\grad} \widetilde{F} : \mathbb{R}^m \to \mathbb{R}^m$, satisfies
                \begin{equation}
                    \label{eq:e_dz}
                    \norm{\overline{\grad} \widetilde{F}(\vb{z}) - \grad \widetilde{F}(\vb{z})}_2 \leq \delta_{\vb{z}}, \quad \forall \vb{z} \in \mathbb{R}^m,
                \end{equation}
                 where $\grad \widetilde{F} : \mathbb{R}^m \to \mathbb{R}^m$ denotes the exact gradient of $\widetilde{F}$, and \( \delta_{\vb{z}} \in \mathbb{R}_{\geq 0} \) indicates the error tolerance at $\vb{z}$. Then,
                \begin{equation*}
                    \widetilde{F}(\vb{x}) \leq \widetilde{F}(\vb{z}) + \overline{\grad} \widetilde{F}(\vb{z})^T(\vb{x}-\vb{z}) + \frac{L'}{2}\norm{\vb{x} - \vb{z}}_2^2 + \delta_{\vb{z}}', \quad \forall \vb{x}, \vb{z} \in \mathbb{R}^m,
                \end{equation*}
                where \( L' = L + 1 \) and \( \delta_{\vb{z}}' = \delta_{\vb{z}}^2 / 2\). 
        \end{lemma}
        \begin{proof}
            Since $\widetilde{F}$ is $L$-smooth, it follows that
            \begin{equation}
                \label{eq:iqub_1}
                \widetilde{F}(\vb{x}) \leq \widetilde{F}(\vb{z}) + \grad \widetilde{F}(\vb{z})^T(\vb{x}-\vb{z}) + \dfrac{L}{2}\norm{\vb{x}-\vb{z}}_2^2, \quad \forall \vb{x} \in \mathbb{R}^m. 
            \end{equation}
            The gradient term in \cref{eq:iqub_1} can be rewritten as 
            \begin{equation}
                \label{eq:iqub_2}
                \grad \widetilde{F}(\vb{z})^T(\vb{x}-\vb{z}) = \overline{\grad}\widetilde{F}(\vb{z})^T(\vb{x}-\vb{z}) + \left[\grad \widetilde{F}(\vb{z}) - \overline{\grad}\widetilde{F}(\vb{z})\right]^T(\vb{x}-\vb{z}).
            \end{equation}
            Now, define $\vb{w}= \grad \widetilde{F}(\vb{z}) - \overline{\grad}\widetilde{F}(\vb{z})$ and $\vb{v}= \vb{x}-\vb{z}$. 
            Since $\norm{\vb{w}-\vb{v}}_{2}^2 \geq 0$ by the non-negativity property of norms, we can deduce that
            \begin{equation}
                \label{eq:iqub_3}
                \left[\grad \widetilde{F}(\vb{z}) - \overline{\grad}\widetilde{F}(\vb{z})\right]^T(\vb{x}-\vb{z}) \leq \dfrac{1}{2}\left[\norm{\grad \widetilde{F}(\vb{z}) - \overline{\grad}\widetilde{F}(\vb{z})}_2^2 + \norm{\vb{x}-\vb{z}}_2^2\right].
            \end{equation}
            Substituting \cref{eq:iqub_3} in \cref{eq:iqub_2} yields the following upper bound for the gradient term:
            \begin{equation}
                \label{eq:iqub_4}
                \grad \widetilde{F}(\vb{z})^T(\vb{x}-\vb{z}) \leq \overline{\grad}\widetilde{F}(\vb{z})^T(\vb{x}-\vb{z}) + \dfrac{1}{2}\norm{\vb{x}-\vb{z}}_2^2 + \dfrac{1}{2}\norm{\grad \widetilde{F}(\vb{z}) - \overline{\grad}\widetilde{F}(\vb{z})}_2^2.
            \end{equation}
            Applying \cref{eq:iqub_4} to \cref{eq:iqub_1}, and subsequently utilizing \cref{eq:e_dz} within the result, yields the desired inequality. 
        \end{proof}
            
        \subsection{Inexact General Descent Method}
            This section introduces an \emph{Inexact General Descent Method} (IGDM) to address smooth unconstrained problems. 
             The method is characterized by the following update rule:
            \begin{equation}
                \label{eq:z_si}
                \vb{z}_{k+1} = \vb{z}_k + t_k\vb{s}_k, \quad \forall k \in \NN_0,
            \end{equation}
            where $(t_k)_{k\in\NN_0}$ represents a positive step size sequence, and \(\vb{s}_k \in \mathbb{R}^m\) indicates an inexact general direction.  The directional conditions we use here originate from classical descent methods such as in \cite{bertsekas_gradient_2000}, but we replace the exact gradients with their inexact counterparts. We assume that there exist constants $c_1' > 0$ and $c_2' > 0$ such that $\vb{s}_k$ satisfies the following conditions:
            \begin{subequations}
                \label{eq:gdi}
                \begin{align}
                    c_1'\norm{\overline{\grad} \widetilde{F}(\vb{z}_k)}_2^2 &\leq - \overline{\grad} \widetilde{F}(\vb{z}_k)^T\vb{s}_k, \quad \forall k \in \NN_0, \label{eq:gdicond_a}\\
                    \intertext{and}
                    \norm{\vb{s}_k}_2 &\leq c_2'\norm{\overline{\grad} \widetilde{F}(\vb{z}_k)}_2, \quad \forall k \in \NN_0. \label{eq:gdicond_b}
                \end{align}
            \end{subequations}
            
            Having introduced the update rule,  the search direction, and the error condition for the inexact gradient, we can consolidate the procedure of the IGDM in \cref{alg:igd_wr}. For this, we consider that the inexact gradient error $\vb{e}_k = \overline{\grad} \widetilde{F}(\vb{z}_k) - \grad \widetilde{F}(\vb{z}_k)$ is such that
            \begin{equation} 
                \label{eq_ek}
                \norm{\vb{e}_k}_2 \leq \delta_k, \quad \forall k \in \mathbb{N}_0,
            \end{equation}
            where \((\delta_k)_{k \in \mathbb{N}_0}\) is a non-negative error tolerance sequence.
            
            \begin{algorithm}[H]
\footnotesize
\caption{\textsc{Inexact General Descent Method}}
\label{alg:igd_wr}
\LinesNotNumbered
\DontPrintSemicolon

\KwData{An initial point $\vb{z}_0 \in \mathbb{R}^m$}
\KwResult{A sequence $\left(\vb{z}_k\right)_{k \in \mathbb{N}_0} \subset \mathbb{R}^m$ whose limit points, if they exist, are stationary}

\BlankLine
\textbf{Step 1.} Set $k \gets 0$.\;
\textbf{Step 2.} Choose $\delta_k \geq 0$.\;
\textbf{Step 3.} Compute $\overline{\grad} \widetilde{F} (\vb{z}_k) \in \mathbb{R}^m$ such that
\[
    \norm{\overline{\grad} \widetilde{F}(\vb{z}_k) - \grad \widetilde{F}(\vb{z}_k)}_2 \leq \delta_k.
\]\;
\textbf{Step 4.} Compute $\vb{s}_k \in \mathbb{R}^m$ so that \cref{eq:gdi} holds.\;
\textbf{Step 5.} Choose $t_k > 0$ by a specific rule.\;
\textbf{Step 6.} Set $\vb{z}_{k + 1} \gets \vb{z}_k + t_k\vb{s}_k$.\;
\textbf{Step 7.} Set $k \gets k + 1$ and \textbf{go to Step 2}.\;
\end{algorithm}

            
            \noindent As will be shown in the next subsection, the method converges provided the sequence $(\delta_k)_{k\in\NN_0}$ satisfies additional conditions that depend on the step–size rule.

            \subsection{Convergence Analysis}
                \label{subsec:conv_anal}
                
                We now turn to the convergence analysis of the IGDM. We begin by presenting several auxiliary results that form the foundation of the proofs. By combining the IGDM update rule with its associated directional conditions, we introduce the following assumption, which will be used throughout the subsequent results:
                \begin{assumption}
                    \label{a.1}
                     Let $\widetilde{F}: \mathbb{R}^m \to \mathbb{R}$ be an $L$-smooth function. Given an initial point $\vb{z}_0 \in \mathbb{R}^m$, consider the sequence $(\vb{z}_k)_{k\in\mathbb{N}}$ generated by \eqref{eq:z_si} employing a positive sequence of step sizes \((t_k)_{k \in \mathbb{N}}\), and a general directions \((\vb{s}_k)_{k\in\mathbb{N}}\) that fulfills \eqref{eq:gdicond_a} and \eqref{eq:gdicond_b} for a pair of constants $c_1' > 0$ and $c_2' > 0$.
                \end{assumption}
                
                Following the classical gradient descent method, our analysis is built upon a key relation between the change in the objective function value across successive iterations and the (inexact) gradient at the current point. This relation, which will be instrumental in establishing the convergence theorems for the step-size strategies considered later, is formalized in the following lemma.
                \begin{lemma}
                    \label{lem:idl}
                     Suppose that \cref{a.1} is satisfied. If \cref{eq_ek} holds, then
                    \begin{equation}
                        \label{eq:idl}
                        \widetilde{F}(\vb{z}_{k+1}) \leq
                        \widetilde{F}(\vb{z}_k) -t_k\left(c_1'-(c_2')^2\dfrac{L't_k}{2}\right)\norm{\overline{\grad} \widetilde{F}(\vb{z}_k)}_2^2 + \delta_k', \quad \forall k \in \mathbb{N}_0,
                    \end{equation}
                    where $L' = L + 1$ and $\delta'_k = \delta_k^2/2$.
                \end{lemma}
                \begin{proof}
                    Since $\widetilde{F}$ is an $L$-smooth function and \cref{eq_ek} holds, we can apply \cref{lem:iqub} with $\vb{x} = \vb{z}_{k+1}$, $\vb{z} = \vb{z}_k$, and $\delta_{\vb{z}_k}' =\delta_k'$. Proceeding in this way, we obtain
                    \begin{align}
                        \widetilde{F}(\vb{z}_{k+1}) &\leq
                        \widetilde{F}(\vb{z}_k) + \overline{\grad} \widetilde{F}(\vb{z}_k)^T(\vb{z}_{k+1}-\vb{z}_k) + \dfrac{L'}{2}\norm{\vb{z}_{k+1}-\vb{z}_k}_2^2 + \delta_k' \nonumber\\
                        &\leq \widetilde{F}(\vb{z}_k) + t_k\overline{\grad} \widetilde{F}(\vb{z}_k)^T\vb{s}_k + \dfrac{L'}{2}t_k^2\norm{\vb{s}_k}_2^2 + \delta_k'. \label{eq:idl_2}
                    \end{align}
                    Using \eqref{eq:gdicond_a} and \eqref{eq:gdicond_b} in \eqref{eq:idl_2} yields \cref{eq:idl}.
            \end{proof}
            \begin{remark} 
                Note that \cref{lem:idl} does not inherently guarantee the monotonic descent property of $\left(\widetilde{F}(\vb{z}_k)\right)_{k\in\NN_0}$, even when $t_k < \frac{c_1'}{(c_2')^2}\frac{2}{L}$ is satisfied for all $k \in \NN_0$, owing to the presence of an additional error term in \cref{eq:idl}. This limitation was previously documented by \cite{devolder_first-order_2014}.
            \end{remark}
    
            When the sequence $(\delta_k)_{k\in\NN_0}$ is defined as a quantity proportional to the gradient or the inexact gradient norm, a common practice in the literature \cite{vasin_accelerated_2023}, it proves useful in ensuring the descent property. The subsequent result formalizes this observation.
    
            
            \begin{lemma}
                \label{lem:stronger}
                Suppose that \cref{a.1} and \eqref{eq_ek} hold. If there exists $\beta \in \left[0, \dfrac{1}{\sqrt{L'}}\dfrac{c_1'}{c_2'}\right)$ such that the sequence $(\delta_k)_{k\in\mathbb{N}_0}$ satisfies
                \begin{equation}
                     \label{eq:stronger_-1}
                    \delta_k \leq \beta\norm{\overline{\grad} \widetilde{F}(\vb{z}_k)}_2,  \quad \forall k \in \mathbb{N}_0,
                \end{equation}
                then
                \begin{enumerate}[label=(\alph*)]
                    \item \label{item:Fz<=}\begin{equation*}
                            \widetilde{F}(\vb{z}_{k+1})\leq \widetilde{F}(\vb{z}_k)- (c_2')^2\dfrac{L'}{2}\left(t_k - \overline{t}_1\right)\left(\overline{t}_2-t_k\right)\norm{\overline{\grad} \widetilde{F}(\vb{z}_k)}_2^2, \quad \forall k \in \mathbb{N}_0,
                        \end{equation*}
                        where
                        \begin{equation*}
                             \overline{t}_1 = \dfrac{c_1'-\sqrt{\Delta}}{(c_2')^2L'}, \quad \text{and} \quad \overline{t}_2 = \dfrac{c_1'+\sqrt{\Delta}}{(c_2')^2L'}, \quad \text{with}\quad \Delta = (c_1')^2 - L'(c_2')^2\beta^2,
                        \end{equation*}
                    \item \label{item:(t1t2)} and   
                            \begin{equation}
                                \label{eq:tk_int}
                               0 \leq \overline{t}_1 < \overline{t}_2 \leq  \dfrac{c_1'}{(c_2')^2}\dfrac{2}{L'}.
                            \end{equation} 
                \end{enumerate}
            \end{lemma}
            \begin{proof}
                Combining \eqref{eq:idl}, the definition of $\delta_k'$ in \cref{lem:idl} and \eqref{eq:stronger_-1} yields
                \begin{align*}
                     \widetilde{F}(\vb{z}_{k+1}) 
                     &\leq \widetilde{F}(\vb{z}_{k}) -\left(-(c_2')^2\dfrac{L'}{2}t_k^2 + c_1't_k -\dfrac{\beta^2}{2}\right)\norm{\overline{\grad} \widetilde{F}(\vb{z}_k)}_2^2\\
                     &\leq \widetilde{F}(\vb{z}_{k}) - (c_2')^2\dfrac{L'}{2}(t_k-\overline{t}_1)(\overline{t}_2-t_k)\norm{\overline{\grad} \widetilde{F}(\vb{z}_k)}_2^2,
                \end{align*}
                which completes the proof of \cref{item:Fz<=}. Now, observe that $\beta \in \left[0, \dfrac{1}{\sqrt{L'}}\dfrac{c_1'}{c_2'}\right)$ implies that $0 < \Delta \leq ( c_1')^2 $. So the roots $\overline{t}_1$ and $\overline{t}_2$ satisfies \eqref{eq:tk_int}. 
            \end{proof}
            \begin{remark}
                \cref{lem:stronger} guarantees monotonic descent if $t_k \in (\overline{t}_1, \overline{t}_2)$ for all $k$. The size of this interval depends on the error: larger $\beta$ (greater error) shrinks the interval, while smaller $\beta$ (less error) widens it. This limits how small the step size can be, unlike in the exact case.
            \end{remark}

            IGDM convergence can be analyzed using either \cref{lem:idl} or \cref{lem:stronger}. With \cref{lem:idl}, convergence is shown for bounded or diminishing step sizes if $(\delta_k)_{k\in\NN_0}$ is quadratically summable or summable, respectively. Without extra conditions, monotonic descent isn't guaranteed. In contrast, \cref{lem:stronger} guarantees monotonic descent, but only for bounded step sizes, since diminishing step sizes strategy certainly would not result in steps within an interval with a positive lower bound.
    
            \subsubsection{Convergence under a Bounded Step Size}

                This subsection proves that, provided the step sizes remain bounded by some specific values, any limit point of the sequence is a stationary point. 
                 
                \begin{theorem}
                    \label{thm:bound_tk}
                    Suppose that \cref{a.1} holds and that $\widetilde{F}$ is bounded below. 
                    Assume further that \cref{eq_ek} is satisfied by a non-negative sequence $(\delta_k)_{k \in \mathbb{N}_0}$ which fulfills \cref{eq:stronger_-1} with 
                    $\beta \in \left[0, \dfrac{1}{\sqrt{L'}}\dfrac{c_1'}{c_2'}\right)$. 
                    Let $\overline{t}_1$ and $\overline{t}_2$ be as in \cref{lem:stronger}, and let $\gamma \in \left(0, \dfrac{\overline{t}_2 - \overline{t}_1}{2}\right]$. 
                    If $t_k \in [\overline{t}_1 + \gamma,\, \overline{t}_2 - \gamma]$ for all $k \in \mathbb{N}_0$, then
                    \begin{enumerate}[label=(\alph*)]
                        \item \label{item:a} $\widetilde{F}(\vb{z}_{k+1}) \leq
                        \widetilde{F}(\vb{z}_k), \quad \forall k \in \mathbb{N}_0$.
                        \item \label{item:b} $\lim\limits_{k \to \infty} \overline{\grad} \widetilde{F}(\vb{z}_k) = \vb{0}_{\mathbb{R}^m}, \quad \text{and} \quad \lim\limits_{k \to \infty} \grad \widetilde{F}(\vb{z}_k) = \vb{0}_{\mathbb{R}^m}$.
                    \end{enumerate}
                    Furthermore, every limit point of $(\vb{z}_k)_{k\in\NN_0}$ is a stationary point of $\widetilde{F}$.
                \end{theorem}
                \begin{proof}
                    Combining \cref{item:Fz<=} of \cref{lem:stronger} with the fact that $t_k \in [\overline{t}_1 + \gamma,\, \overline{t}_2 - \gamma]$,
                    we have that
                    \begin{equation}
                        \label{eq:bound_tk1}
                        \widetilde{F}(\vb{z}_{k+1})\leq \widetilde{F}(\vb{z}_k)- \dfrac{L'(c_2')^2 \gamma^2}{2} \norm{\overline{\grad} \widetilde{F}(\vb{z}_k)}_2^2\leq \widetilde{F}(\vb{z}_k), \quad \forall k \in \mathbb{N}_0.
                    \end{equation}
                     Summing  from $0$ to $K - 1$, with $K \geq 1$, results in
                    \begin{equation}
                        \label{eq:bound_tk4}
                        \widetilde{F}(\vb{z}_K) \leq \widetilde{F}(\vb{z}_0) - \dfrac{L'(c_2')^2 \gamma^2}{2}\sum\limits_{k = 0}^{K-1}\norm{\overline{\grad}\widetilde{F}(\vb{z}_k)}_2^2.
                    \end{equation}
                    Given that $\widetilde{F}$ is bounded below, it follows from \cref{eq:bound_tk4} that
                    \begin{equation}
                        \label{eq:bound_tk5}
                        \sum\limits_{k = 0}^{\infty}\norm{\overline{\grad}\widetilde{F}(\vb{z}_k)}_2^2 < \infty.
                    \end{equation}
                    Consequently,
                    \begin{equation}
                        \label{eq:bound_tk6}
                    \lim_{k\to\infty} \overline{\grad}\widetilde{F}(\vb{z}_k) = \vb{0}_{\RR^m}.
                    \end{equation}
                    Combining \cref{eq_ek,eq:stronger_-1,eq:bound_tk6} yields
                    \begin{equation}
                        \label{eq:bound_tk7}
                        \lim_{k\to\infty} \grad\widetilde{F}(\vb{z}_k) = \vb{0}_{\RR^m}.
                    \end{equation}
                    This concludes the proof of \cref{item:b}. By continuity, every limit point of $(\vb{z}_k)_{k\in\NN_0}$ is a stationary point of $\widetilde{F}$.
                \end{proof}
            
            \subsubsection{Convergence under Diminishing Step Size}

                We now analyze convergence under diminishing step sizes. 
                In contrast to the bounded (constant) step-size regime, which often requires tuning using problem-specific quantities such as a Lipschitz constant, we focus on a positive sequence $(t_k)_{k\in\NN_0}$ that satisfies
                \begin{align}
                     \label{tk_c1}
                    \sum\limits_{k=0}^{\infty} t_k^2 &< \infty,
                    \intertext{and}
                    \label{tk_c2}
                    \sum\limits_{k=0}^{\infty} t_k &= \infty.
                \end{align}
                These classical conditions guarantee sustained progress while steps vanish, thereby mitigating the need for precise knowledge of problem constants.

                For this step size strategy, our proof relies on the general argument presented in \cite{bertsekas_nonlinear_2016}, when analyzing the exact version of the method. In this context, the following technical auxiliary lemma is crucial.
    
                \begin{lemma}
                    \label{lemma:isl}
                    Suppose that \cref{a.1} holds, and that \eqref{eq_ek} is satisfied by a non-negative, summable sequence $(\delta_k)_{k \in \mathbb{N}_0}$. If
                    \begin{align}
                         \sum\limits_{k = 0}^{\infty} t_k\norm{\overline{\grad} \widetilde{F}(\vb{z}_k)}_2^2 < \infty, \label{eq:lim_norm_gine_0}
                         \intertext{and}
                         \lim\limits_{k \to \infty} \inf \norm{\overline{\grad}\widetilde{F}(\vb{z}_k)}_2 = 0, \label{eq:lim_norm_gine_1}
                    \end{align}
                    then
                    \begin{equation}
                        \label{eq:lim_norm_gine_r}
                        \lim\limits_{k \to \infty} \norm{\overline{\grad}\widetilde{F}(\vb{z}_k)}_2 = 0.
                    \end{equation}
                \end{lemma}
                \begin{proof}
                    In order to derive a contradiction, suppose that \eqref{eq:lim_norm_gine_r} fails to hold. Consequently, there exists $\epsilon > 0$ such that
                     \begin{equation}
                        \label{eq:lim_inf_3}
                        \lim\limits_{k \to \infty} \sup \norm{\overline{\grad} \widetilde{F}(\vb{z}_k)}_2 \geq \epsilon.
                    \end{equation}
                   Since \cref{eq:lim_norm_gine_1,eq:lim_inf_3} hold, the sequence of inexact gradient norms exhibit an oscillatory pattern, which allows the partition of the sequence into segments characterized by relatively small or large norms. Accordingly, define $(m_j)_{j \in \mathbb{N}_0}$ and $(n_j)_{j \in \mathbb{N}_0}$ as sequences of indices such that
                    \begin{equation*}
                                m_j < n_j < m_{j + 1}, \quad \forall j \in \mathbb{N}_0,
                    \end{equation*} 
                    and
                    \begin{align}
                        \norm{\overline{\grad}\widetilde{F}(\vb{z}_k)}_2 > \dfrac{\epsilon}{3}, & \quad \text{for } m_j \leq k < n_j, \label{eq:lim_inf_4}\\[10pt]
                        \norm{\overline{\grad}\widetilde{F}(\vb{z}_k)}_2 \leq \dfrac{\epsilon}{3}, & \quad \text{for } n_j \leq k < m_{j + 1}. \label{eq:lim_inf_5}
                    \end{align}
                    Since $(\delta_k)_{k\in\mathbb{N}_0}$ is summable, given $\dfrac{\epsilon}{12}>0$ there exist a sufficiently large constant $j_1 \in \mathbb{N}_0$ such that
                    \begin{equation}
                        \label{eq:lim_norm_gine_2}
                        \sum\limits_{k = m_{j_1}}^{\infty}\delta_k < \dfrac{\epsilon}{12}.
                    \end{equation}
                    Analogously, based on \eqref{eq:lim_norm_gine_0}, 
                    it is possible to determine a sufficiently large constant $\mathbb{N}_0 \ni j_2 > j_1 $ such that
                    \begin{equation}
                        \label{eq:lim_norm_gine_3}
                        \sum\limits_{k = m_{j_2}}^{\infty}
                        t_k\norm{\overline{\grad}\widetilde{F}(\vb{z}_k)}_2^2 < \dfrac{\epsilon^2}{18L}.
                    \end{equation}
                    
                    Now, note that
                    \begin{align}
                        \norm{\overline{\grad}\widetilde{F}(\vb{z}_{k+1}) - \overline{\grad}\widetilde{F}(\vb{z}_{k})}_2 
                        &\leq \norm{\overline{\grad}\widetilde{F}(\vb{z}_{k+1}) - \grad\widetilde{F}(\vb{z}_{k+1})}_2 + \norm{\grad\widetilde{F}(\vb{z}_{k+1})-\grad\widetilde{F}(\vb{z}_{k})}_2 \nonumber\\
                        & \qquad \qquad + \norm{\overline{\grad}\widetilde{F}(\vb{z}_{k}) - \grad\widetilde{F}(\vb{z}_{k})}_2\nonumber\\
                        &\leq \delta_{k+1} + L\norm{\vb{z}_{k+1}-\vb{z}_k}_2 + \delta_k.  \label{eq:lim_norm_gine_4} 
                    \end{align}
                    Consequently, for every $j \in \mathbb{N}_0$, where $j \geq j_2$, and each $m \in \mathbb{N}_0$ satisfying $m_j \leq m \leq n_j - 1$, it follows from \cref{eq:lim_norm_gine_2,eq:lim_norm_gine_4} that
                    \begin{align}
                        \norm{\overline{\grad}\widetilde{F}(\vb{z}_{n_{j}}) - \overline{\grad}\widetilde{F}(\vb{z}_{m})}_2 &\leq \sum\limits_{k = m}^{n_j-1}\norm{\overline{\grad}\widetilde{F}(\vb{z}_{k+1}) - \overline{\grad}\widetilde{F}(\vb{z}_{k})}_2\nonumber\\
                        &\leq \sum\limits_{k = m}^{n_j-1}\delta_{k+1}+ L \sum\limits_{k = m}^{n_j-1} \norm{\vb{z}_{k+1}-\vb{z}_k}_2  + \sum\limits_{k = m}^{n_j-1}\delta_k\nonumber\\
                        &= \dfrac{\epsilon}{6} + L \sum\limits_{k = m}^{n_j-1} t_k \norm{\overline{\grad}\widetilde{F}(\vb{z}_{k})}_2. \label{eq:diff_grad}
                    \end{align}
                    By exploiting \cref{eq:lim_inf_4,eq:lim_norm_gine_3}, the right-hand side of \cref{eq:diff_grad} simplifies to
                    \begin{align}
                        \norm{\overline{\grad}\widetilde{F}(\vb{z}_{n_{j}}) - \overline{\grad}\widetilde{F}(\vb{z}_{m})}_2 &\leq \dfrac{\epsilon}{6} + \dfrac{3L}{\epsilon} \sum\limits_{k = m}^{n_j-1} t_k \norm{\overline{\grad}\widetilde{F}(\vb{z}_{k})}_2^2\nonumber\\
                        &\leq \dfrac{\epsilon}{6} + \dfrac{3L}{\epsilon}\dfrac{\epsilon^2}{18L}\nonumber\\
                        &=\dfrac{\epsilon}{3}. \label{eq:lim_norm_gine_5} 
                    \end{align}
                    Based on \eqref{eq:lim_inf_5} and \eqref{eq:lim_norm_gine_5}, it can be established that
                    \begin{align}
                        \norm{\overline{\grad} \widetilde{F}(\vb{z}_m)}_{2} 
                        &\leq \norm{\overline{\grad} \widetilde{F}(\vb{z}_m) - \overline{\grad} \widetilde{F}(\vb{z}_{n_j})}_2 + \norm{\overline{\grad} \widetilde{F}(\vb{z}_{n_j})}_2 \nonumber\\
                        &\leq \dfrac{\epsilon}{3} + \dfrac{\epsilon}{3} \nonumber\\
                        &= \dfrac{2\epsilon}{3}, \quad \forall j \geq j_2, \quad m_j \leq m \leq n_j - 1.  \label{eq:lim_inf_8}
                    \end{align}
                    By combining \cref{eq:lim_inf_5,eq:lim_inf_8}, it is immediate that
                    \begin{equation*} 
                          \norm{\overline{\grad} \widetilde{F}(\vb{z}_m)}_{2} \leq \dfrac{2\epsilon}{3},  \quad \forall j \geq j_2, \quad m_j \leq m < m_{j + 1}
                    \end{equation*}
                    which means that
                    \begin{equation*}
                        \norm{\overline{\grad} \widetilde{F}(\vb{z}_m)}_{2} \leq \dfrac{2\epsilon}{3}, \quad \forall m \geq m_{j_2}. 
                    \end{equation*}
                    However, this contradicts \eqref{eq:lim_inf_3}. Therefore, \eqref{eq:lim_norm_gine_r} must holds.
                \end{proof}
                
                Utilizing \cref{lemma:isl}, we can complete the convergence proof for the diminishing step-size scenario, as detailed in the theorem below.
                \begin{theorem}
                    \label{thm:c_a_ds}
                    Suppose that \cref{a.1} holds and that $\widetilde{F}$ is bounded below. 
                    Assume further that \eqref{eq_ek} is satisfied by a non-negative, summable sequence $(\delta_k)_{k \in \mathbb{N}_0}$. 
                    If the sequence $(t_k)_{k \in \mathbb{N}_0}$ is diminishing in accordance with \eqref{tk_c1} and \eqref{tk_c2}, then
                    \begin{equation*}
                        \lim_{k \to \infty} \overline{\grad} \widetilde{F}(\vb{z}_k) = \vb{0}_{\mathbb{R}^m}, \quad \text{and} \quad\lim_{k \to \infty} \grad \widetilde{F}(\vb{z}_k) = \vb{0}_{\mathbb{R}^m}.
                    \end{equation*}
                    Furthermore, every limit point of $(\vb{z}_k)_{k\in\NN_0}$ is a stationary point of $\widetilde{F}$.
                \end{theorem}
                \begin{proof}
                    By \cref{lem:idl}, we have that
                    \begin{equation}
                        \label{thm:conv_ineg_pd0}
                        \widetilde{F}(\vb{z}_{k+1}) \leq
                        \widetilde{F}(\vb{z}_k) -t_k\left(c_1'-(c_2')^2\dfrac{L't_k}{2}\right)\norm{\overline{\grad} \widetilde{F}(\vb{z}_k)}_2^2 + \delta_k', \quad \forall k \in \mathbb{N}_0.
                    \end{equation}
                    Since \((t_k)_{k \in \mathbb{N}_0}\) is quadratically summable, it converges to zero. Thus, for a given $\widetilde{c} \in (0, c_1')$ there exists a constant \(k_1 \in \mathbb{N}\) such that
                    \begin{align}
                        t_k \leq \frac{2(c_1'-\widetilde{c})}{(c_2')^2L'}, \quad \forall k \geq k_1.\nonumber
                        \intertext{Rearranging terms yields}
                        -\widetilde{c} \geq -\left(c_1' - (c_2')^2\dfrac{L' t_k}{2}\right), \quad \forall k \geq k_1.  \label{thm:conv_ineg_pd1}
                    \end{align}
                    Using \eqref{thm:conv_ineg_pd1} in \eqref{thm:conv_ineg_pd0}, we arrive at
                    \begin{equation}
                        \label{thm:conv_ineg_pd2}
                         \widetilde{F}(\vb{z}_{k+1}) \leq \widetilde{F}(\vb{z}_k) -\widetilde{c}t_k\norm{\overline{\grad} \widetilde{F}(\vb{z}_k)}_2^2 + \delta'_{k}, \quad \forall k \geq k_1.
                    \end{equation}
                    By summing \cref{thm:conv_ineg_pd2} over $k = k_1, \ldots, K-1$, with $K-1 > k_1$, we obtain
                    \begin{equation}
                        \label{thm:conv_ineg_pd3}
                        \widetilde{F}(\vb{z}_{K}) \leq \widetilde{F}(\vb{z}_{k_1}) + \dfrac{1}{2}\sum\limits_{k = k_1}^{K-1} \delta_k^2 - \widetilde{c} \sum\limits_{k = k_1}^{K-1} t_k\norm{\overline{\grad} \widetilde{F}(\vb{z}_k)}_2^2.
                    \end{equation}
                    As $(\delta_k)_{k \in \mathbb{N}_0}$ is nonnegative and summable, it follows that it is also quadratically summable. Thus, using the fact that $\widetilde{F}$ is bounded below, we obtain  that
                    \begin{equation}
                        \label{thm:conv_ineg_pd5}
                        \sum\limits_{k = k_1}^{\infty} t_k\norm{\overline{\grad} \widetilde{F}(\vb{z}_k)}_2^2 < \infty.
                    \end{equation}
                    
                    Now, suppose for contradiction that there exist a constant $\epsilon > 0$, and $k_2 \in \mathbb{N}$ such that
                    \begin{equation*}
                        \norm{\overline{\grad} \widetilde{F}(\vb{z}_k)}_2 \geq \epsilon, \quad \forall k \geq k_2.
                    \end{equation*}
                    Thus, since $(t_k)_{k \in \mathbb{N}_0}$ is not summable, we have
                    \begin{equation*}
                        \sum\limits_{k = k_2}^{\infty} t_k\norm{\overline{\grad} \widetilde{F}(\vb{z}_k)}_2^2 \geq \epsilon^2 \sum \limits_{k=k_2}^{\infty}t_k = \infty.
                    \end{equation*}
                    However, this contradicts condition \cref{thm:conv_ineg_pd5}. Hence,
                    \begin{equation}
                        \label{thm:conv_ineg_pd6}
                        \lim \inf\limits_{k \to \infty} \norm{\overline{\grad}\widetilde{F}(\vb{z}_k)}_2 = 0.
                    \end{equation}
                    By virtue of conditions \cref{thm:conv_ineg_pd5,thm:conv_ineg_pd6}, \cref{lemma:isl} ensures that
                    \begin{equation}
                         \label{thm:conv_ineg_pd7}
                        \lim\limits_{k \to \infty} \overline{\grad}\widetilde{F}(\vb{z}_k) = \vb{0}_{\mathbb{R}^m}.
                    \end{equation}
                    From the triangle inequality and \cref{eq_ek}, we obtain
                    \begin{align}
                        \norm{\grad \widetilde{F}(\vb{z}_k)}_2 
                        &\leq \norm{\overline{\grad} \widetilde{F}(\vb{z}_k)}_2 + \norm{\grad \widetilde{F}(\vb{z}_k) - \overline{\grad} \widetilde{F}(\vb{z}_k)}_2 \nonumber \\
                        &\leq \norm{\overline{\grad} \widetilde{F}(\vb{z}_k)}_2 + \delta_k. \label{thm:conv_ineg_pd8}
                    \end{align}
                    As both terms on the right-hand side of \eqref{thm:conv_ineg_pd8} approach zero as $k \to \infty$,  the proof is completed.
                \end{proof}

            \subsection{Computation of the General Direction }

                Having established the convergence  analysis of \cref{alg:igd_wr}, attention now turns to its practical implementation. Fortunately, a procedure can be formulated for selecting a general direction $\vb{s}_k$ that inherently fulfills the conditions in \cref{eq:gdi}. The following theorem formalizes this result. The statement parallels the sufficient-descent results for Newton-type directions in \citet{forsgren_gill_murray_1995}, adapted here to allow inexact gradients on the right-hand side and to ensure positive definiteness via a uniformly bounded spectrum.
                
                \begin{theorem}
                    \label{thm:sk_dir}
                    Let $\widetilde{F}:\RR^m \to \RR$ be a differentiable function. Given an initial point $\vb{z}_0 \in \mathbb{R}^m$, consider the sequence $(\vb{z}_k)_{k \in \mathbb{N}}$ generated by \cref{eq:z_si} using a positive sequence of step sizes $(t_k)_{k \in \mathbb{N}_0}$, and a general direction $\vb{s}_k \in \mathbb{R}^m$ that satisfies
                    \begin{equation*}
                        \vb{H}_k \vb{s}_k = -\,\overline{\grad}\widetilde{F}(\vb{z}_k), \quad \forall k \in \NN_0,
                    \end{equation*}
                    where $\vb{H}_k \in \mathbb{R}^{m\times m}$ is a real symmetric with a uniformly bounded spectrum matrix, which means that there exist constants $\mu,\eta>0$ such that
                    \begin{equation*}
                        0 < \mu \le \lambda_i\!\left(\vb{H}_k\right) \le \eta, \quad \forall i \in \{1,\ldots,m\}, \ \forall k \in \NN_0,
                    \end{equation*}
                    where $\lambda_i(\vb{H}_k)$ denotes the $i$th eigenvalue of $\vb{H}_k$. Then the conditions \cref{eq:gdicond_a,eq:gdicond_b} are valid for some pair of constants $c_1'>0$ and $c_2'>0$.
                \end{theorem}
                \begin{proof}
                    See Lemma~2.3 in \cite{forsgren_gill_murray_1995}. 
                \end{proof}

\section{Adjoint-Based Gradient Error}\label{subsec:adj_error}
The convergence results in \cref{subsec:conv_anal} are based on the assumption of a bounded inexact-gradient error \cref{eq_ek}. 
            This condition holds for several sensitivity schemes, including finite differences, linear interpolation, and Gaussian and spherical smoothing \cite{berahas_theoretical_2022}. 
            Importantly, this requirement is also met by the discrete adjoint method, which we use in this work to evaluate the sensitivities of the objective functions in our study cases. Below, we outline why the adjoint-based gradient meets the bounded-error condition, drawing on the framework of \cite{brown_effect_2022}, which analyzes how inexact state and adjoint solves propagate to gradient inaccuracies in the discrete adjoint method.

            Suppose that a state equation of the form given by \cref{eq:ge_sys} is approximately solved in each iteration as follows:
            \begin{align}
                \norm{\vb{R}\left(\vb{z}_k,\overline{\vb{u}(\vb{z}_k)}\right)}_2 &\leq \tau_{\mathcal{R},k}, \quad \forall k \in \NN_0, \label{eq:R_tol}
                \intertext{where $\overline{\vb{u}(\vb{z}_k)} \in \RR^n$ is an approximation for $\vb{u}(\vb{z}_k)$, and $\tau_{\mathcal{R},k} > 0$ is an adaptive state tolerance. Similarly, given these approximate states, assume that the adjoint equation in \cref{eq:adj_eq} is solved inexactly, namely,}\norm{\vb{R}_{\psi}\left(\overline{\bm{\psi}}_k;\vb{z}_k, \overline{\vb{u}(\vb{z}_k)}\right)}_2 &\leq \tau_{\psi,k}, \quad \forall k \in \NN_0, \label{eq:R_psi_tol}
            \end{align}
            where $\overline{\bm{\psi}}_k\in \RR^n$ denotes an approximation of the exact solution of the adjoint equation with inexact state variables, and $\tau_{\psi,k} > 0$ is an adaptive adjoint tolerance. Using both the inexact state and adjoint variables obtained from \cref{eq:R_tol,eq:R_psi_tol}, we obtain the following inexact gradient from \cref{eq:opt_nand}:
            \begin{equation}
                \label{eq:ine_grad}
                \overline{\grad} \widetilde{F}(\vb{z}_k) = 
                \grad_{\vb{z}} F \left(\vb{z}_k,\overline{\vb{u}(\vb{z}_k)}\right) -  
                \left[\pdv{\vb{R}}{\vb{z}}\bigg\rvert_{\left(\vb{z}_k,\overline{\vb{u}(\vb{z}_k)}\right)}\right]^\top\overline{\bm{\psi}_k}, \quad \forall k \in \NN_0.
            \end{equation}  
            For linear state equations as described by \cref{eq:ge_sys}, it can be shown from the theoretical framework established by \cite{brown_effect_2022} that the inexact gradient presented in \cref{eq:ine_grad} satisfies
            \begin{equation}
                \label{eq:da_bound}
                \norm{\overline{\grad}\widetilde{F}(\vb{z}_k)-\grad\widetilde{F}(\vb{z}_k)}_2 \leq C'\tau_{\mathcal{R},k} + C''\tau_{\psi,k}, \quad \forall k \in \mathbb{N}_0,
            \end{equation}
            where $C' > 0$ and $C'' > 0$ are constants that depend on the Lipschitz continuity and boundedness of the partial Jacobians of $F$ and $\vb{R}$ with respect to both the control and state variables. Thus, if we define the state and adjoint tolerances as
            \begin{equation}
                \label{eq:tau_r_psi}
                \tau_{\mathcal{R},k} \leq \gamma_1 \norm{\overline{\grad}\widetilde{F}(\vb{z}_k)}_2,  \quad \text{and} \quad
                \tau_{\psi,k} \leq \gamma_2\norm{\overline{\grad}\widetilde{F}(\vb{z}_k)}_2, \quad \forall k \in \NN_0,
            \end{equation}
            where $\gamma_1 \in \left(0,\dfrac{\beta}{2C'}\right]$ and $\gamma_2 \in \left(0,\dfrac{\beta}{2C''}\right]$ are defined constants for a given  $\beta>0$, it follows from \cref{eq:da_bound,eq:tau_r_psi} that
            \begin{align}
                 \norm{\overline{\grad}\widetilde{F}(\vb{z}_k)-\grad\widetilde{F}(\vb{z}_k)}_2 
                 &\leq \left(C'\dfrac{\beta}{2C'} + C''\dfrac{\beta}{2C''}\right)\norm{\overline{\grad}\widetilde{F}(\vb{z}_k)}_2\nonumber\\
                 &\leq \beta \norm{\overline{\grad}\widetilde{F}(\vb{z}_k)}_2, \quad \forall k \in \NN_0. \label{eq:ine_grad_err}
            \end{align}
            The established condition corresponds with that outlined in \cref{eq:stronger_-1} if $\beta$ is small enough. Consequently, we can ensure the convergence of the IGDM by employing a suitable constant step size, with gradients computed inexactly through the discrete adjoint method.

            To enforce \cref{eq:ine_grad_err}, the inexact gradient at iteration $k\ge0$ is computed via a discrete-adjoint test-and-tighten scheme (\cref{alg:aisara}) with a state–adjoint nesting. Because the target bounds in \cref{eq:da_bound} depend on the (unknown) inexact gradient, we start with trial tolerances and tighten them geometrically as needed. Given $\vb{z}_k$, the outer loop solves the state equation \eqref{eq:R_tol} up to a trial target $\tau_{\mathcal R,k}^{(i)}$.  An inner loop solves the adjoint equation \eqref{eq:R_psi_tol} up  to $\tau_{\psi,k}^{(j)}$ and compute the inexact gradient, and it is repeated tightening $\tau_{\psi,k}^{(j)}$ until the adjoint residual meets the bound proportional to $\gamma_2$ times the inexact gradient. Then we check if the state residual meets the required bound, proportional to $\gamma_1$ times the inexact gradient.  If the criterion is satisfied, the approximate value of the state variables is set, otherwise $\tau_{\mathcal R,k}^{(i)}$ is tightened and the outer loop repeats. The procedure terminates whenever the exact gradient is not null, ensuring that both bounds in \eqref{eq:R_tol} and \eqref{eq:R_psi_tol} are satisfied simultaneously for
                \begin{equation} \label{eq:def_taus}
                    \tau_{\mathcal R,k}=\gamma_1\norm{\overline{\grad}\widetilde F(\vb z_k)}_2,
                \quad \text{and} \quad
                \tau_{\psi,k}=\gamma_2\norm{\overline{\grad}\widetilde F(\vb z_k)}_2,
                \end{equation} which guarantees \cref{eq:ine_grad_err}.  The trial tolerances $\tau_{\mathcal R,k}^{(i)}$ and $\tau_{\psi,k}^{(j)}$ are recognized only as iterates used to discover these levels. 
                
                As we can see, the test-and-tighten scheme allows us to ensure that the inexact gradient error satisfies \cref{eq:ine_grad_err}. This differs from the approach in \cite{brown_effect_2022}, where the authors employ a purely heuristic strategy, fixing the tolerances in the solution of the state and adjoint equations to be proportional to the inexact gradient norm at the previous iterate.

                \newpage
                \begin{algorithm}[H]
                    \small
                    \DontPrintSemicolon
                    \caption{\textsc{Adaptive Inexact State–Adjoint Residual Algorithm}}
                    \label{alg:adap}
                    \KwData{Current iterate $\vb{z}_k$; previous gradient norm $\norm{\overline{\grad} \widetilde{F}(\vb{z}_{k-1})}_2$; 
        initial tolerances $\tau_{\mathcal{R}}^{(0)}$, $\tau_{\psi}^{(0)}$; 
        lower bounds $\underline{\tau_{\mathcal{R}}}$, $\underline{\tau_{\psi}}$; 
        and parameters $\gamma_1 \in \left(0,\dfrac{\beta}{2C'}\right]$, $\gamma_2 \in \left(0,\dfrac{\beta}{2C''}\right]$.}
                    \label{alg:aisara}
                    Set $i \gets 0$ and define the state tolerance as\;
                    \If{$k = 0$}{
                        $\tau_{\mathcal{R},k}^{(i)} \gets \tau_{\mathcal{R}}^{(0)}$\;
                    }
                    \Else{
                        $\tau_{\mathcal{R},k}^{(i)} \gets 
                        \min\!\left(0.5 \cdot \gamma_1 
                        \norm{\overline{\grad} \widetilde{F}(\vb{z}_{k-1})}_2,
                        \tau_{\mathcal{R}}^{(0)} \right)$\;
                    }
                    \Repeat{$\norm{\vb{R}\left(\vb{z}_k, \overline{\vb{u}(\vb{z}_k)}\right)}_2 \leq \gamma_1\norm{\overline{\grad} \widetilde{F}(\vb{z}_k)}_2$}{
                         Compute $\overline{\vb{u}(\vb{z}_k)} \in\RR^n$ by solving $\norm{\vb{R}\left(\vb{z}_k,\overline{\vb{u}(\vb{z}_k)}\right)}_2 \leq \tau_{\mathcal{R}, k}^{(i)}$\;
                         Set $j \gets 0$ and the adjoint tolerance\;
                         \If{$k = 0$}{
                                $\tau_{\psi,k}^{(j)} \gets \tau_{\psi}^{(0)}$\;
                            }
                            \Else{
                                $\tau_{\psi,k}^{(j)} \gets 
                                \min\!\left(0.5 \cdot \gamma_2 
                                \norm{\overline{\grad} \widetilde{F}(\vb{z}_{k-1})}_2,
                                \tau_{\psi}^{(0)} \right)$\;
                            }
                        \Repeat{$\norm{\vb{R}_{\psi}\left(\overline{\boldsymbol{\psi}}_k;\vb{z}_k, \overline{\vb{u}(\vb{z}_k)}\right)}_2 \leq \gamma_2\norm{\overline{\grad} \widetilde{F}(\vb{z}_k)}_2$} {
                           Compute $\overline{\bm{\psi}}_k\in \RR^n$ by solving $\norm{\vb{R}_{\psi}\left(\overline{\bm{\psi}}_k;\vb{z}_k, \overline{\vb{u}(\vb{z}_k)}\right)}_2 \leq \tau_{\psi, k}^{(j)}$\;
                            $\overline{\grad} \widetilde{F}(\vb{z}_k) \gets \grad_{\vb{z}} F\left(\vb{z}_k, \overline{\vb{u}(\vb{z}_k)}\right) - \left[\pdv{\vb{R}}{\vb{z}}\Big\vert_{\left(\vb{z}_k, \overline{\vb{u}(\vb{z}_k)}\right)}\right]^\top\overline{\boldsymbol{\psi}}_k$\;
                            $j \gets j + 1$\;
                            Update: $\tau_{\psi, k}^{(j)} \gets\num{0.1}\cdot\gamma_2\norm{\overline{\grad} \widetilde{F}(\vb{z}_k)}_2$
                         }
                         $i \gets i + 1$\;
                         Update: $\tau_{\mathcal{R}, k}^{(i)} \gets \num{0.5}\cdot\gamma_1\norm{\overline{\grad} \widetilde{F}(\vb{z}_k)}_2$
                    }
                \end{algorithm} 

\section{Numerical Results}\label{sec:num_res}

    This section presents numerical experiments involving the DECO problems described in \cref{sec:std_cas}. The objective is to assess the computational savings obtained by using inexact adjoint gradients with adaptive tolerances, as opposed to fixed, restrictive tolerances.  In addition, we aim to determine whether  incorporating a BFGS-like update for additional flexibility in the search direction provides further computational benefits within the inexact framework. 
    To achieve these objectives, we implemented two instances of the Inexact General Descent Method (IGDM) in C++, computing inexact gradients with the adaptive procedure in \cref{alg:adap}, which keeps the gradient error tolerance proportional to the inexact gradient. Since convergence has been demonstrated under this error tolerance condition for bounded step sizes, we adopted a constant, data-aware step size that scales inversely with the number of sampled reference points. The specific step-size rule used in each problem setup is discussed in the following subsections.

    In the first instance, we consider an Inexact Gradient Descent (IGD) method defined by the search direction
    \[
        \vb{s}_k^{\mathrm{IGD}} = -\,\overline{\grad}\,\widetilde{F}(\vb{z}_k), \quad k \in \mathbb{N}_0.
    \]
    In this configuration, the inexact adjoint gradients are computed using \cref{alg:adap} with adaptive tolerance parameters \(\gamma_{1}=\SI{60}{}\) and \(\gamma_{2}=\SI{3}{}\). These parameters were chosen through preliminary grid searches and demonstrated stable performance across problem sizes, and they are consistently used in both case studies presented in this section. For reference, we also implemented a conventional Gradient Descent (GD) method, in which adjoint gradients are computed using fixed tight tolerances \(\tau_{\mathcal{R},k} = \tau_{\psi,k} = 10^{-9}\) for all \(k \in \mathbb{N}_0\), and termination occurs when the norm of the exact gradient falls below the same threshold.

    To investigate whether curvature information could provide additional savings beyond the inexact gradient scheme, we implemented a second instance of IGDM that updates an inverse-Hessian approximation through a BFGS-type recursion. 
    Starting with $\vb{H}_0^{-1} = \vb{I}_{m \times m}$, this approximation is updated at each iteration by
    \begin{equation*}
        \vb{H}_{k+1}^{-1}
        =
        \bigl(\vb{I}_{m \times m} - \rho_k \vb{p}_k \vb{y}_k^{\top}\bigr)
        \vb{H}_k^{-1}
        \bigl(\vb{I}_{m \times m} - \rho_k \vb{y}_k \vb{p}_k^{\top}\bigr)
        + \rho_k \vb{p}_k \vb{p}_k^{\top},
    \end{equation*}
    where
    \begin{equation*}
        \rho_k = \frac{1}{\vb{y}_k^{\top} \vb{p}_k}, \qquad 
        \vb{p}_k = \vb{z}_{k+1} - \vb{z}_k, \qquad
        \vb{y}_k = \overline{\grad}\,\widetilde{F}(\vb{z}_{k+1}) - \overline{\grad}\,\widetilde{F}(\vb{z}_k).
    \end{equation*}
    The corresponding search direction is then given by
    \[
        \vb{s}_k^{\mathrm{IBFGS}} = -\,\vb{H}_k^{-1}\,\overline{\grad}\,\widetilde{F}(\vb{z}_k).
    \]
    IGDM with this choice of direction yields an inexact BFGS-like method (IBFGS). 
    Due to the inexact nature of the gradient, we apply a few stabilization measures to ensure robustness. First, the adaptive tolerance parameters \(\gamma_1\) and \(\gamma_2\) are tuned separately for each case study to maintain reliable curvature updates and mitigate instabilities. In addition, a Powell-type damping is employed to preserve the symmetry and positive definiteness of the inverse-Hessian approximations \cite{powell_algorithms_1978}.

    As noted before, both IGD and IBFGS were implemented under the same adaptive tolerance framework given by \cref{alg:aisara}. A practical aspect of this implementation concerns the choice of its input parameters. The trial targets start from fixed initial tolerances $\tau_{\mathcal R}^{(0)}>0$ and $\tau_{\psi}^{(0)}>0$ and are progressively tightened during the iterations. To prevent oversolving as these targets decrease, we enforce lower bounds $\underline{\tau}_{\mathcal R}$ and $\underline{\tau}_{\psi}$, both set to $10^{-9}$ in this work. Regarding the coefficients $\gamma_1$ and $\gamma_2$, we selected their values empirically based on short preliminary runs, since their theoretical upper bounds depend on Lipschitz-type constants that are typically unknown or unreliable in practice.

    The numerical experiments for both instances of IGDM are presented separately in \cref{subsec:sec_ode} and \cref{subsec:lap_pde}, corresponding to the two case studies considered. All computations were performed on a system equipped with an Intel Core i7-8550U CPU running at \SI{1.80}{\giga\hertz}, 8~GB of RAM, and a \SI{512}{\giga\byte} SSD, using Ubuntu~24.04.2~LTS.

        \subsection{Case Study 1: Second-Order Inverse Problem}
            \label{subsec:sec_ode}

            For the second-order inverse problem described in \cref{subsec:cs_1}, both the state and adjoint equations were solved using in-house C++ implementations based on a successive over-relaxation (SOR) scheme. 
            Since the corresponding linear systems are Toeplitz tridiagonal, we used a relaxation factor close to its optimal value in the SOR iterations. 
            For the inverse reconstruction, we uniformly sampled \(n_P = 12\) mesh indices to form the set \(S\). 
            At these locations, we evaluated the analytical solution \(u(x; \vb{z})\) with the reference control vector 
            \(\vb{z}_{\mathrm{ref}} = 
            \begin{bmatrix}
            1 & -3 & -4 & 1 & 1 & 5 & 0 & 0
            \end{bmatrix}^{\top}\)
            to obtain the reference state values. 
            To ensure consistency across mesh refinements, we anchored the sampling locations so that approximately the same physical coordinates were used to compute the reference values for all tested values of \(M\). 
            Preliminary tests indicated that the optimizers operated effectively across these mesh sizes using a step size \(t = 1.15/n_P = \SI{9.583e-2}{}\). 
            The initial guess was 
            \(\vb{z}_0 = 
            \begin{bmatrix} 
            1.5 & 2.0 & -7.0 & 0.2 & -0.4 & 0.1 & 1.0 & -0.2
            \end{bmatrix}^{\top}\),
            and the stopping criterion was defined as the inexact gradient norm falling below the tolerance \(\epsilon = \SI{1e-6}{}\). 
            These settings were applied consistently across all methods. 
            It is worth noting that all optimizations successfully converged.

            To evaluate the numerical performance under these settings, we analyzed wall-clock optimization times and the associated computational savings across different mesh resolutions. 
            For each grid size \(M\) from 16 to 128, we measured total optimization time for GD and IGD and presented the corresponding runtime curves, comparing the results across small, medium, and large instances. 
            In addition, we reported the percentage-reduction curve of IGD relative to GD for \(M = 64\) through \(128\), where the gap between the methods became more pronounced, and included a horizontal line indicating the average percentage reduction over this interval. 
            We then used the IGD curve obtained from this first comparison as a baseline to benchmark the IBFGS method, applying the same range of problem sizes and reporting optimization time as a function of \(M\). 
            Finally, we also included the percentage-reduction curve of IBFGS relative to IGD, together with the average reduction over the range, which quantified the computational savings attributable to curvature information in the inexact framework.


            The results of these comparisons are summarized in \cref{fig:cs1_runtimes}, which presents wall-clock optimization time (seconds) as a function of \(M\) for GD, IGD, and IBFGS, together with the corresponding percentage-reduction curves. 
On average, IGD reduced runtime by \(\SI{45.55}{\percent}\) relative to GD, while IBFGS achieved an average reduction of \(\SI{87.69}{\percent}\) relative to IGD. 
Thus, IBFGS consistently outperformed IGD, as expected. 
Aside from a few outliers, the pointwise reductions clustered near these averages across \(M\), indicating robust behavior. 
Further exploratory repeats with the same settings, not shown here, confirmed the same qualitative behavior. 
For IBFGS, we used \(\gamma_{1}=\SI{1e-3}{}\) and \(\gamma_{2}=\SI{1e-3}{}\). 
With these settings, curvature safeguards, the descent-check fallback, and bad-curvature events were rare across \(M\), and no resets to the identity were observed in the reported runs. 
We also verified that the counts of state and adjoint solves closely tracked the iteration count across \(M\), indicating minimal overhead from the adaptive tolerance mechanism prescribed by \cref{alg:adap}.

            \newpage
            \begin{figure}
    \centering
    \begin{subfigure}{0.48\textwidth}
        \centering
        \includegraphics[width=\linewidth]{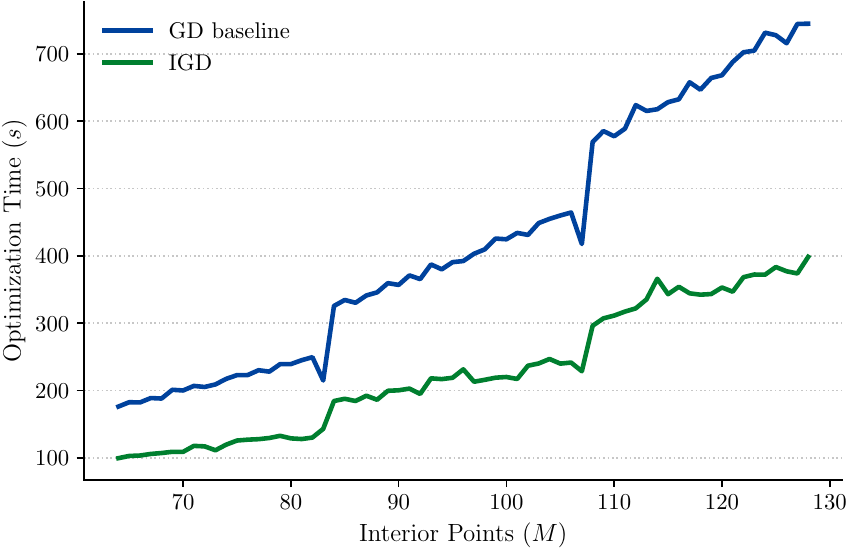}
        \caption{IGD compared with GD baseline.}
        \label{fig:cs1_igd_gamma1}
    \end{subfigure}\hfill
    \begin{subfigure}{0.5\textwidth}
        \centering
        \includegraphics[width=\linewidth]{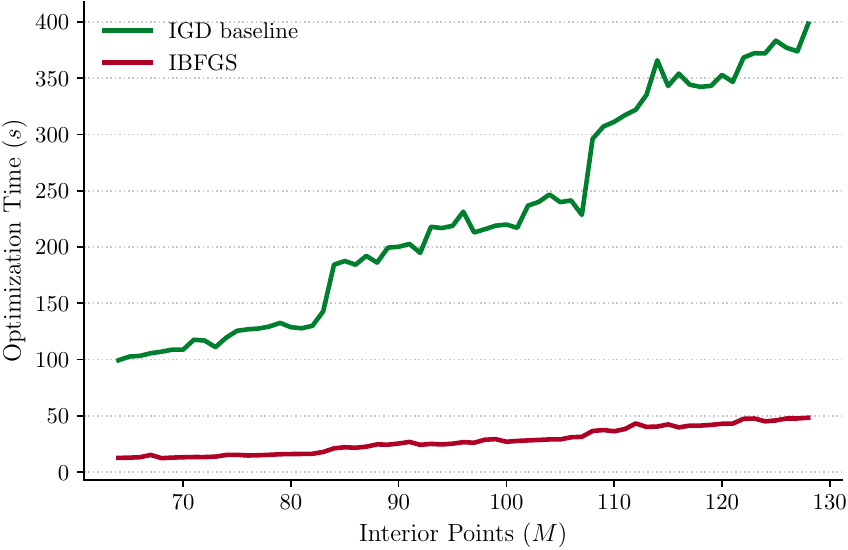}
        \caption{IBFGS compared with IGD baseline.}
        \label{fig:cs1_igd_gamma2}
    \end{subfigure}
\end{figure}

\begin{figure}
    \ContinuedFloat
    \centering
    \begin{subfigure}{0.48\textwidth}
        \centering
        \includegraphics[width=\linewidth]{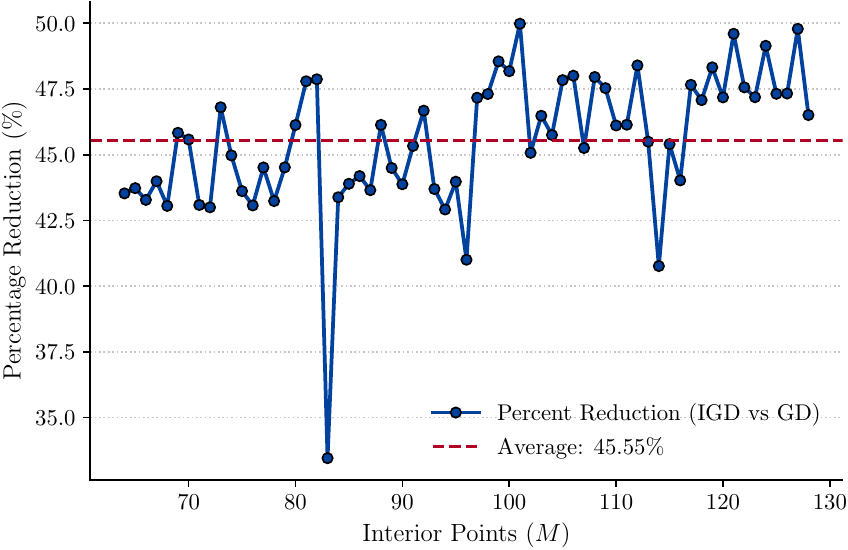}
        \caption{Percent reduction: IGD relative to GD.}
        \label{fig:cs1_ibfgs_gamma1}
    \end{subfigure}\hfill
    \begin{subfigure}{0.5\textwidth}
        \centering
        \includegraphics[width=\linewidth]{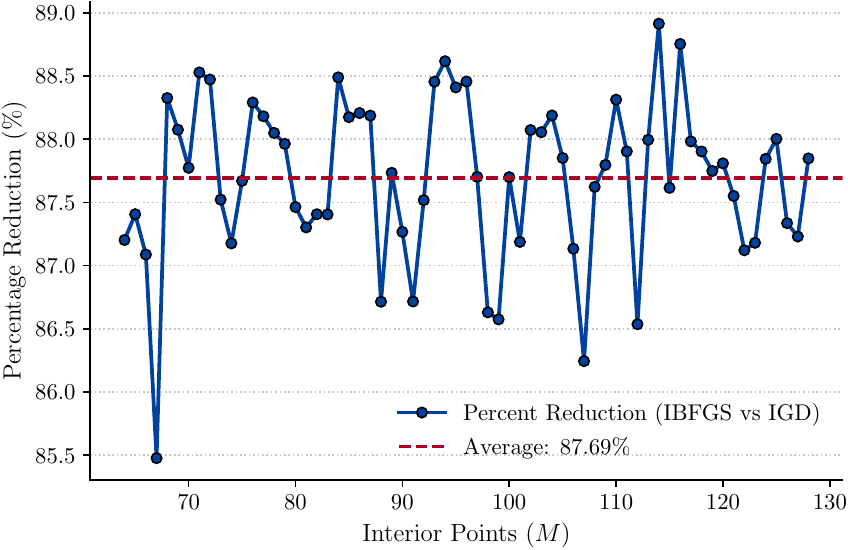}
        \caption{Percent reduction: IBFGS relative to IGD.}
        \label{fig:cs1_ibfgs_gamma2}
    \end{subfigure}
    \caption{Optimization time versus interior points \(M\) with percentage-reduction curves. IGD uses \(\gamma_{1}=\SI{60}{}\), \(\gamma_{2}=\SI{3}{}\), and IBFGS uses \(\gamma_{1}=\SI{1e-3}{}\), \(\gamma_{2}=\SI{1e-3}{}\). All methods use constant step size \(t=\SI{9.583e-2}{}\) and stopping tolerance \(\epsilon=\SI{1e-6}{}\). Horizontal lines show the mean reduction for \(64 \le M \le 128\).}
    \label{fig:cs1_runtimes}
\end{figure}

        \subsection{Case Study 2: Laplace Inverse Problem}
            \label{subsec:lap_pde}

            For the Laplace inverse problem described in \cref{subsec:cs_2}, the state equation was solved using an alternating direction implicit (ADI) scheme, where the tridiagonal subsystems in each coordinate direction were handled via Thomas’ algorithm. The adjoint equation was solved using restarted GMRES with sparse matrix–vector products and SOR preconditioning. In contrast to Case Study~1, where the number of sampled measurements was fixed, here we scaled the number of samples with the grid size to provide more data to fit and to increase the effective complexity of the reconstruction as the mesh was refined. This design is meant to test whether the algorithm can preserve its efficiency under these more demanding conditions. Specifically, for each grid size \(M\) we drew \(n_s = 4(M+1)+2\) mesh indices by Latin hypercube sampling to form \(\mathcal{I}(\Omega_d)\), and at these locations we evaluated the analytical solution at 
            \(\vb{z}_{\text{ref}} = \begin{bmatrix} 0 & 1 & -1 \end{bmatrix}^{\!\top}\) to obtain the reference state values. Since \(n_s\) increases with \(M\), we used the grid-dependent step size \(t = 3.5/(n_s - 2)\), which empirically maintained stability and yielded consistent efficiency across mesh sizes. The initial guess was \(\vb{z}_0 = \begin{bmatrix} 0.5 & 1.0 & 0.25 \end{bmatrix}^{\top}\), and the stopping criterion was set to the inexact gradient norm falling below \(\epsilon = 10^{-6}\). These configurations were uniformly applied across all methods.

            
            The same comparison procedure described in Case Study~1 was adopted here to evaluate performance across methods and mesh resolutions. 
            The results are summarized in \cref{fig:cs2_runtimes}, which presents wall-clock optimization time as a function of \(M\) for GD, IGD, and IBFGS, together with the corresponding percentage-reduction curves. 
            The qualitative pattern resembled that of Case Study~1, with one key difference: the saving for IGD over GD was much larger here, averaging \(\SI{90.21}{\percent}\), whereas the additional saving of IBFGS over IGD averaged \(\SI{79.61}{\percent}\). 
            Although the latter was smaller than in Case Study~1, it remained substantial. 
            These percentage reductions translated into larger absolute time savings than in Case Study~1, particularly for the GD-to-IGD comparison. 
            For IBFGS, we set \(\gamma_{1} = \SI{1e-1}{}\) and \(\gamma_{2} = \SI{1e-1}{}\), following the same rationale as in the previous case. 
            Pointwise reductions remained close to the averages across \(M\), confirming robust behavior as before. 
            Safeguard triggers and descent-check fallbacks were infrequent, no identity resets were observed, and the counts of state and adjoint solves closely tracked the iteration counts across \(M\), again indicating minimal overhead from the adaptive tolerance mechanism prescribed by \cref{alg:adap}.
           \begin{figure}[H]
    \centering
    \begin{subfigure}{0.48\textwidth}
        \centering
        \includegraphics[width=\linewidth]{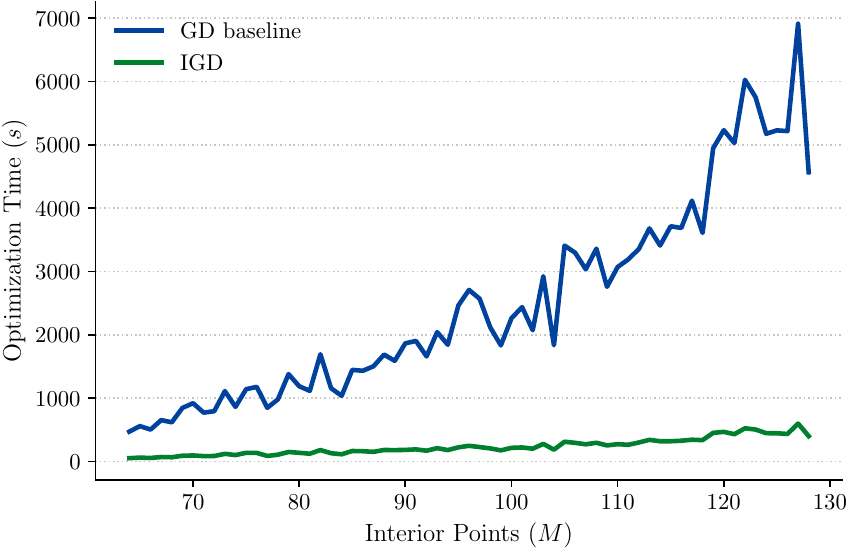}
        \caption{IGD compared with GD baseline.}
        \label{fig:cs2_igd_gamma1}
    \end{subfigure}\hfill
    \begin{subfigure}{0.5\textwidth}
        \centering
        \includegraphics[width=\linewidth]{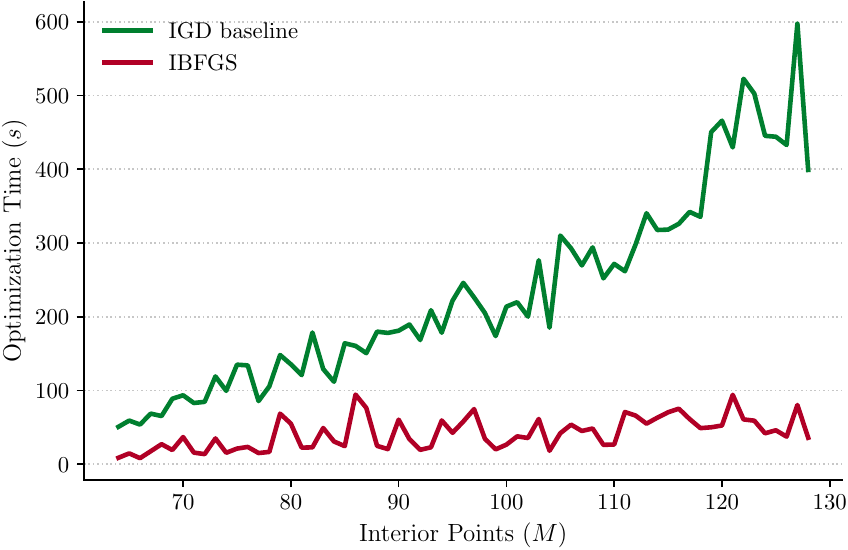}
        \caption{IBFGS compared with IGD baseline.}
        \label{fig:cs2_igd_gamma2}
    \end{subfigure}

    \vskip\baselineskip

    \begin{subfigure}{0.48\textwidth}
        \centering
        \includegraphics[width=\linewidth]{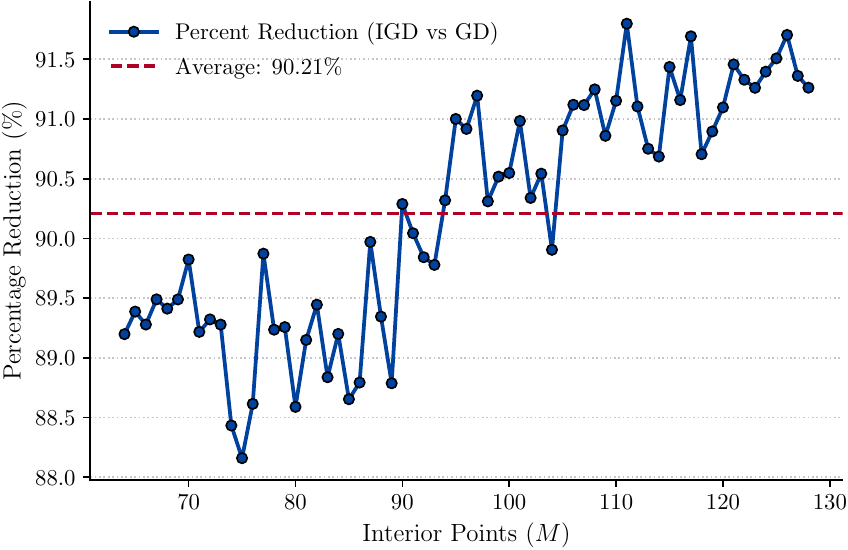}
        \caption{Percent reduction: IGD relative to GD.}
        \label{fig:cs2_ibfgs_gamma1}
    \end{subfigure}\hfill
    \begin{subfigure}{0.5\textwidth}
        \centering
        \includegraphics[width=\linewidth]{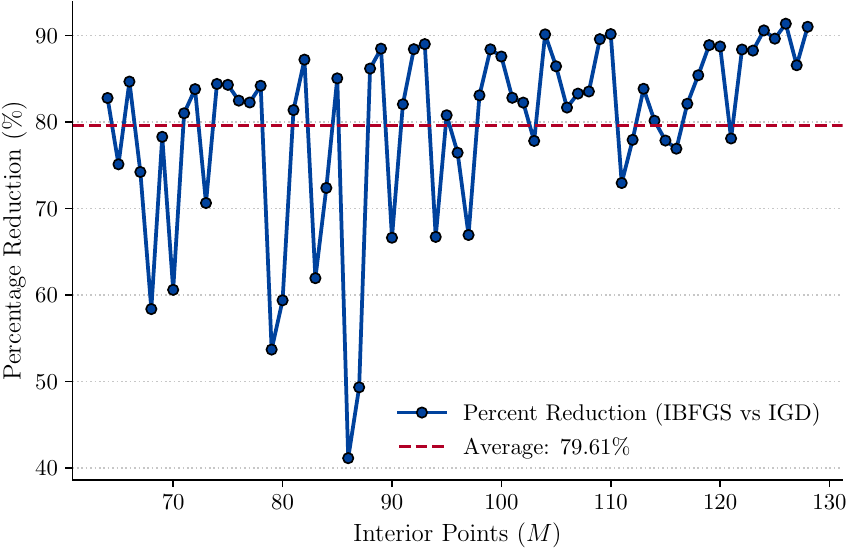}
        \caption{Percent reduction: IBFGS relative to IGD.}
        \label{fig:cs2_ibfgs_gamma2}
    \end{subfigure}

    \caption{Wall-clock optimization time as a function of interior points \(M\), together with percentage-reduction curves. IGD uses \(\gamma_{1}=\SI{60}{}\) and \(\gamma_{2}=\SI{3}{}\). IBFGS uses \(\gamma_{1}=\SI{1e-1}{}\) and \(\gamma_{2}=\SI{1e-1}{}\). All methods use stopping tolerance \(\epsilon=\SI{1e-6}{}\). Horizontal lines indicate the mean reduction for \(64 \le M \le 128\).}
    \label{fig:cs2_runtimes}
\end{figure}
            \newpage


\section{Final Remarks and Future Research}\label{sec:fin_rem}
    Motivated by DECO applications, this work introduced a novel inexact-gradient-based framework with general inexact search directions for smooth unconstrained optimization. A comprehensive convergence analysis was developed for this framework under both bounded and diminishing step-size rules. For bounded step sizes, we identified an admissible interval that ensures convergence while guaranteeing a monotonic-descent property when the error-tolerance sequence is proportional to the norm of the computed inexact gradient. For diminishing step sizes, convergence was achieved under summable error-tolerance sequences, although monotonic descent could not be ensured in this case. Furthermore, the proposed framework combines theoretical rigor and practical implementability by allowing flexible search directions and by defining acceptance conditions that depend directly on the inexact gradient rather than on its unavailable exact counterpart.

    Two instances of the framework were implemented and tested with constant step sizes: the Inexact Gradient Descent (IGD) and the Inexact BFGS-like (IBFGS) methods. These were applied to a second-order inverse problem and to a two-dimensional Laplace inverse problem. In both applications, gradients were computed via an adaptive adjoint test-and-tighten scheme, which adjusts the state and adjoint tolerances dynamically to produce gradients with controlled accuracy. The use of adaptive inexact gradients led to significant runtime reductions compared to fixed-tolerance baselines, while incorporating curvature information through the IBFGS scheme provided additional computational savings. Both IGD and IBFGS exhibited consistent performance across mesh sizes, with negligible overhead from the adaptive tolerance mechanism. For IBFGS, the implemented safeguards effectively maintained a well-conditioned inverse-Hessian approximation.

    Despite these encouraging results, several limitations remain. The study cases considered residual functions that are linear in the state variables, ensuring globally well-defined state mappings. A natural continuation is to address nonlinear residuals, where such mappings may fail to exist globally. The present framework also focuses solely on unconstrained problems, and its extension to include equality and inequality constraints remains an open and practically relevant challenge. Moreover, the numerical experiments were restricted to the bounded step-size regime because only in this setting we could prove convergence for error tolerances proportional to the inexact gradient norm. Finally, the current implementation does not employ a line-search procedure, as inexact objective evaluations make its direct use nontrivial.
    
    Future research will primarily focus on overcoming these limitations. We plan to broaden the theoretical analysis to nonlinear residual and constrained problems and to design line-search or trust-region mechanisms that incorporate precision-control explicitly. In addition, an important theoretical direction involves developing conditions that allow the use of diminishing step sizes even when error tolerances remain proportional to the inexact gradient. Beyond these extensions, further research avenues include the use of Inexact Restoration techniques to globalize the optimization process, allowing the algorithm to follow infeasible paths with respect to the state variable.

\backmatter


\section*{Statements and Declarations}

\bmhead{Competing interests}
The authors have no relevant financial or non-financial interests to disclose.

\bmhead{Funding}
This work was supported by the Brazilian Federal Agency for Support and Evaluation of Graduate Education (CAPES, grant no.~88887.669708/2022-00) and by the São Paulo Research Foundation (FAPESP, grants no.~2013/07375-0, 2022/05803-3, and 2023/08706-1).








\bmhead{Supplementary Information}
Animations illustrating the two study cases are available at:  
\url{https://youtu.be/X905iImwfmE} and
\url{https://youtu.be/onk9BWbxR3s}.

\bibliography{sn-bibliography}

\end{document}